\documentclass[11pt]{amsart}
\pdfoutput=1
\usepackage[margin=1in]{geometry}
\usepackage{amsthm,amsmath,amssymb,enumitem,hyperref,fontenc}
\usepackage[utf8]{inputenc}

\usepackage{tikz-cd}
\usepackage{stackengine}
\usepackage{trimclip}

\newtheorem{theorem}{Theorem}[section]
\newtheorem{definition-theorem}[theorem]{Definition-Theorem}
\newtheorem{lemma}[theorem]{Lemma}
\newtheorem{corollary}[theorem]{Corollary}
\newtheorem{proposition}[theorem]{Proposition}

\theoremstyle{definition}
\newtheorem{definition}[theorem]{Definition}
\newtheorem{example}[theorem]{Example}

\newtheorem{remark}[theorem]{Remark}
\newtheorem{notation}[theorem]{Notation}

\newcommand{\xto}[1]{\xrightarrow{#1}}
\newcommand{\cover}{\lessdot}
\newcommand{\cP}{\mathcal{P}}
\newcommand{\cA}{\mathcal{A}}
\newcommand{\cQ}{\mathcal{Q}}
\newcommand{\cL}{\mathcal{L}}
\newcommand{\cX}{\mathcal{X}}

\newcommand{\cS}{\mathcal{S}}
\newcommand{\cT}{\mathcal{T}}
\newcommand{\cE}{\mathcal{E}}
\newcommand{\cF}{\mathcal{F}}

\newcommand{\field}{\mathbb{K}}

\newcommand{\fG}{\mathfrak{G}}
\newcommand{\fH}{\mathfrak{H}}
\newcommand{\cG}{\mathcal{G}}
\newcommand{\add}{\mathrm{add}}
\newcommand{\gldim}{\mathrm{gl.dim}}
\DeclareMathOperator{\rep}{\mathrm{rep}}
\DeclareMathOperator{\repdim}{\mathrm{repdim}}
\DeclareMathOperator{\Rep}{\mathrm{Rep}}
\DeclareMathOperator{\colim}{\mathrm{colim}}
\DeclareMathOperator{\Hom}{\mathrm{Hom}}
\DeclareMathOperator{\End}{\mathrm{End}}
\DeclareMathOperator{\Ker}{\mathrm{Ker}}
\DeclareMathOperator{\mods}{\mathrm{mod}}
\DeclareMathOperator{\ind}{\mathrm{ind}}
\DeclareMathOperator{\proj}{\mathrm{proj}}
\DeclareMathOperator{\coker}{\mathrm{coker}}
\DeclareMathOperator{\im}{\mathrm{im}}

\newcommand{\tail}{\mathrel{\clipbox*{{-.5ex} {-.5ex} {.3\width} {1.5ex}}{$\rightarrowtail$}}}

\newcommand\mono[2][]{%
  \mathrel{\ooalign{$\mkern-4mu\tail$\cr%
$\mkern4mu\xrightarrow[#1]{#2}$\cr}}
}

\newcommand\epi[2][]{%
  \mathrel{\ooalign{$\xrightarrow[#1\mkern4mu]{#2\mkern4mu}$\cr%
  \hidewidth$\rightarrow\mkern4mu$}}
}

\makeatletter
\providecommand*{\twoheadrightarrowfill@}{%
  \arrowfill@\relbar\relbar\twoheadrightarrow
}
\providecommand*{\twoheadleftarrowfill@}{%
  \arrowfill@\twoheadleftarrow\relbar\relbar
}
\providecommand*{\xtwoheadrightarrow}[2][]{%
  \ext@arrow 0579\twoheadrightarrowfill@{#1}{#2}%
}
\providecommand*{\xtwoheadleftarrow}[2][]{%
  \ext@arrow 5097\twoheadleftarrowfill@{#1}{#2}%
}
\makeatother

\title{Exact Structures for Persistence Modules}

\author{Benjamin Blanchette}
\address{Benjamin Blanchette, Départment de Mathématiques,
Université de Sherbrooke,
Sherbrooke, QC, J1K 2R1, Canada}
\email{Benjamin.Blanchette@USherbrooke.ca}

\author{Thomas Brüstle}
\address{Thomas Brüstle, Départment de Mathématiques,
Université de Sherbrooke,
Sherbrooke, QC, J1K 2R1, Canada and
 Bishop’s University, Sherbrooke, QC, J1M 1Z7, Canada}
\email{Thomas.Brustle@USherbrooke.ca, tbruestl@bishops.ca}

\author{Eric J. Hanson}
\address{Département de Mathématiques, Université de Sherbrooke, Sherbrooke, QC, J1K 2R1, Canada;\newline \indent Département de Mathématiques, LACIM, Université du Québec à Montréal, Montréal, QC,\newline\indent\indent H2L 2C4, Canada}
\email{ejhanso3@ncsu.edu}



\begin{document}

\maketitle

\begin{abstract}
 We discuss applications of exact structures and relative homological algebra to the study of invariants of multiparameter persistence modules. This paper is mostly expository, but does contain a pair of novel results. Over finite posets, classical arguments about the relative projective modules of an exact structure make use of Auslander-Reiten theory. One of our results establishes a new adjunction which allows us to ``lift'' these arguments to certain infinite posets over which Auslander-Reiten sequences do not always exist. We give several examples of this lifting, in particular highlighting the non-existence and existence of resolutions by upsets when working with finitely presentable representations of the plane and of the closure of the positive quadrant, respectively. We then restrict our attention to finite posets. In this setting, we discuss the relationship between the global dimension of an exact structure and the representation dimension of the incidence algebra of the poset. We conclude with our second novel contribution. This is an explicit description of the irreducible morphisms between relative projective modules for several exact structures which have appeared previously in the literature.
\end{abstract}

\tableofcontents

\section{Introduction}

The classical approach to topological data analysis is to pass from a data set (e.g. a finite set of points in $\mathbb{R}^n$) to a representation of a totally ordered set (called a \emph{1-parameter persistence module}) by building a filtration of topological spaces and taking homology. By an extension of Gabriel's theorem for representations of quivers of type A (see \cite[Theorem~1.2]{BCB}), the resulting representation uniquely decomposes into a direct sum of ``interval modules''. These intervals encode the ``birth points'' and ``death points'' of homological features, and together comprise the \emph{barcode} of the persistence module, and can be seen as encoding the ``birth points'' and ``death points'' of homological features. As the barcode uniquely determines the isomorphism class of the persistence module, the problem of understanding 1-parameter persistence modules is essentially solved from a representation-theoretic perspective.

On the other hand, there are many instances in topological data analysis where one wishes to build a filtration of topological spaces using two or more parameters. See e.g. \cite[Section~1.1]{BL} for examples. This leads to the concept of \emph{multiparameter persistent homology}, originally formulated in \cite{CZ}. In this framework, one obtains a representation of a product of totally ordered sets. With few exceptions, these posets have wild representation type, meaning that an explicit classification of all indecomposable representations lies far beyond computational feasibility. One approach to this problem is to describe persistence modules using \emph{additive invariants} such as the dimension vector, the barcode, or the rank invariant of a persistence module. We defer a more formal discussion and examples to Section~\ref{sec:invariant}.

Recently, it has become apparent that tools from relative homological algebra (RHA), and specifically the theory of \emph{exact categories}, are useful in understanding old and defining new invariants. 
In particular, many classical invariants from topological data analysis  (the barcode, the dimension vector, the rank invariant, and many of their generalizations) can be formulated 
nicely using the language of exact categories,
see \cite{asashiba,AENY,BBH,BOO,BOOS,CGRST,OS}. Following the third author's presentation of the paper \cite{BBH} at the 2022 International Conference on Representations of Algebras (ICRA), the present paper is meant to serve as a complement to these works. Indeed, a beautiful summary of the study of multiparameter persistence is included in the proceedings of the 2020 ICRA conference \cite{BL}. Moreover, \cite[Section~1 and~3]{BBH} outlines motivations for using RHA to define new invariants and explains its relationship to other constructions in the literature. Thus to avoid redundancy, this paper has three main goals.

\begin{enumerate}
    \item Many of the arguments appearing in the study of exact structures of persistence modules are taken from the classical papers \cite{AS,DRSS}. Both of these papers assume the category being studied has \emph{almost split sequences}, a property which holds when working over finite posets but not necessarily over infinite ones. As observed in \cite{BOO,BOOS,OS}, many of the arguments can indeed be extended to posets which do not have almost split sequences, but we believe there is some subtlety to be addressed. This is mostly done in Sections~\ref{sec:exact} and~\ref{sec:infinite}. In particular, we show in Examples~\ref{ex:infinite_exact_2} and~\ref{ex:upsets} that finitely presentable modules over $\mathbb{R}^2$ cannot always be resolved by finitely presentable upset-representations, while those over $[0,\infty)^2 \subseteq \mathbb{R}^2$ can.
    \item The ``homological invariants'' defined in \cite{BBH} require the global dimension of an exact structure to be finite. Over finite posets, this is often proved by showing that the endomorphism ring of the relative projectives has finite global dimension. It was asked at the 2022 ICRA conference whether this is related to the notion of \emph{representation dimension}. Indeed, in \cite{AENY}, the theory of ``quasi-hereditary algebras'' is used to prove the finiteness of a particular exact structure arising in persistence theory. Such algebras were fundamental in Iyama's proof of the finiteness of the representation dimension \cite{iyama} (see also \cite{ringel}). We expand upon this discussion in Section~\ref{sec:rep_dim}.
    \item In \cite[Section~8.3]{OS} and \cite[Sections~6 and~7]{BOOS}, a recipe is given (and its hypotheses verified in special cases) for establishing stability results of homological invariants for $n$-parameter persistence modules. One of the key ingredients is to show that, when working over finite subgrids (see Defintion~\ref{def:grid}), the global dimension of an exact structure is uniformly bounded with respect to the size of the grid. In Section~\ref{sec:irreducible}, we describe the irreducible morphisms between indecomposable relative projectives over certain exact structures (over finite posets). This in particular allows us to compute explitly the quiver of the corresponding endomorphism ring, which we hope to be a useful tool for establishing similar stability results in future works.
\end{enumerate}


\section{Preliminaries}\label{sec:background}

In this section, we recall background information about posets, persistence modules, and invariants. We work over a fixed poset $\cP$ throughout this section.

\subsection{Posets and lattices}

Fix a poset $\cP$. A subset $A \subseteq \cP$ is called an \emph{antichain} if there are no relations between distinct elements of $A$. If $A = \{a\}$ is a singleton, we will often simply denote $A = a$. For $A$ and $B$ two antichains, we write $A \leq B$ if there exists $a \in A$ and $b \in B$ such that $a \leq b$. 
Influenced by that appearing in \cite{BBH,BOO}, we fix the following notation and terminology.

\begin{definition}\label{def:segment_hook}
    Let $A \leq B$ be antichains. Denote
    \begin{enumerate}
        \item $\langle A,\infty\langle \ = \{c \in \cP \mid A \leq c\}$. Subsets of this form are referred to as \emph{closed upsets}.
        \item $\langle A,B\rangle = \{c \in \cP \mid A \leq c \leq B\}$. If $|A| = 1 = |B|$, we call $\langle A,B\rangle$ a \emph{segment}.
        \item $\langle A,B\langle \ = \{c \in \cP \mid A \leq c \not\geq B\}$. If $|A| = 1 = |B|$, we call $\langle A,B\langle$ a \emph{hook}.
        \item $\rangle A,B\rangle = \{c \in \cP \mid A \not \geq c \leq B\}$. If $|A| = 1 = |B|$, we call $\langle A,B\langle$ a \emph{cohook}.
    \end{enumerate}
\end{definition}

See Figure~\ref{fig:spreads} for examples of these constructions.

\begin{figure}
    \begin{tikzpicture}[scale = 0.7]
        \filldraw (0,0) circle (0.1);
        \filldraw (1,0) circle (0.1);
        \filldraw (2,0) circle (0.1);
        \filldraw (0,1) circle (0.1);
        \filldraw (1,1) circle (0.1);
        \filldraw (2,1) circle (0.1);
        \filldraw (0,2) circle (0.1);
        \filldraw (1,2) circle (0.1);
        \filldraw (2,2) circle (0.1);
        \draw[fill,color=black,opacity=.2] (-0.5,1.5)--(-0.5,2.5)--(2.5,2.5)--(2.5,-0.5)--(1.5,-0.5)--(1.5,0.5)--(0.5,0.5)--(0.5,1.5)--cycle;

        \node[anchor = north] at (0,2) {$a_1$};
        \node[anchor = north] at (1,1) {$a_2$};
        \node[anchor = north] at (2,0) {$a_3$};
        \node[anchor = north] at (0,1) {$c_1$};
        \node[anchor = north] at (1,0) {$c_2$};
        \node[anchor = south] at (2,2) {$d$};

        \node at (1,-1) {(i)};
    \begin{scope}[shift = {(5,0)}]
                \filldraw (0,0) circle (0.1);
        \filldraw (1,0) circle (0.1);
        \filldraw (2,0) circle (0.1);
        \filldraw (0,1) circle (0.1);
        \filldraw (1,1) circle (0.1);
        \filldraw (2,1) circle (0.1);
        \filldraw (0,2) circle (0.1);
        \filldraw (1,2) circle (0.1);
        \filldraw (2,2) circle (0.1);
        \draw[fill,color=black,opacity=.2] (0.5,-0.5)--(0.5,2.5)--(2.5,2.5)--(2.5,-0.5)--cycle;
        \node [anchor = north] at (1,0) {$a$};
        \node [anchor = north] at (0,0) {$c$};
        \node [anchor = south] at (2,2) {$b$};

        \node at (1,-1) {(ii)};
    \end{scope}
      \begin{scope}[shift = {(10,0)}]
        \filldraw (0,0) circle (0.1);
        \filldraw (1,0) circle (0.1);
        \filldraw (2,0) circle (0.1);
        \filldraw (0,1) circle (0.1);
        \filldraw (1,1) circle (0.1);
        \filldraw (2,1) circle (0.1);
        \filldraw (0,2) circle (0.1);
        \filldraw (1,2) circle (0.1);
        \filldraw (2,2) circle (0.1);
        \draw[fill,color=black,opacity=.2] (-0.5,-0.5)--(-0.5,2.5)--(0.5,2.5)--(0.5,0.5)--(2.5,0.5)--(2.5,-0.5)--cycle;
        \node [anchor = north] at (0,0) {$a$};
        \node [anchor = south] at (1,1) {$b$};
        \node [anchor = south] at (2,1) {$c$};

        \node at (1,-1) {(iii)};
    \end{scope}
    \end{tikzpicture}
    \caption{(i) The closed upset $\langle \{a_1,a_2,a_3\},\infty\langle$. This can also be written as $\rangle \{c_1,c_2\},d\rangle$. (ii) The segment $\langle a,b\rangle$, the closed upset $\langle a,\infty\rangle$, and the cohook $\rangle c,b\rangle$ all coincide. (iii) The hook $\langle a,b\langle$. All examples are taken over the poset $\{0,1,2\}\times \{0,1,2\}$ with the usual product order.}\label{fig:spreads}
\end{figure}

\begin{definition}\label{def:covers}
    Let $\cS \subseteq \cP$ and $x \in \cP$. We say that $\cS$ \emph{covers} $x$, in symbols $x \cover \cS$, if the following both hold.
    \begin{enumerate}
        \item $x \leq \cS$ and $x \notin \cS$.
        \item If $y \in \cP$ is such that $x \lneq y \leq \cS$, then $y \in \cS$.
    \end{enumerate}
    We define $\cS \cover x$ analogously. Note also that if $\cS = \{z\}$ is a singleton, then $x \cover \{z\}$ if and only if $z$ covers $x$ in the classical sense. In this case, we often write $x \cover z$ in place of $x \cover \{z\}$.
\end{definition}

\begin{example}
    Let $\cS$ be the shaded region of Figure~\ref{fig:spreads}(iii). Then $\cS \cover b$. Moreover, $\cS \not\hspace{-0.3em}\cover\hspace{0.3em} c$ but there exists $y \in \cS$ such that $y \cover c$.
\end{example}

\begin{definition}\label{def:spread}
    Let $\cS \subseteq \cP$.
    \begin{enumerate}
        \item We say that $\cS$ is an \emph{upset} if for all $x \in \cS$ and $x \leq y \in \cP$, one has $y \in \cS$.
        \item We say that $\cS$ is \emph{convex} if for all $x \leq y \in \cS$, one has $\langle x,y\rangle \subseteq \cS$.
        \item We say that $\cS$ is \emph{connected} if for all $x, y \in \cS$, there exists a sequence
        $x = z_0 \leq z_1 \geq z_2 \leq \cdots \geq z_k = y$
        with each $z_i \in \cS$.
        \item We say that $\cS$ is a \emph{spread} if it is both convex and connected\footnote{Spreads are often referred to as ``intervals'' in the literature, but we avoid using this term to avoid confusion with segments. Note also that in \cite{BBH}, the term ``spread'' means only a convex subset and ``connected spread'' is used for the concept defined here.}.
        \item We say that $\cS$ is an \emph{fp-spread} if it is a spread of the form $\langle A,B\langle$ or $\langle A,\infty\langle$ for $A$ and $B$ finite antichains in $\cP$. We also call $\langle A,\infty\langle$ an \emph{fp-upset}.
        \item We say that $\cS$ is a \emph{single-source fp-spread} if it is of the form $\langle x,\infty\langle$ or $\langle x,B\langle$ for $x \in \cP$ and $B \subseteq \cP$ a finite antichain.
    \end{enumerate}
\end{definition}

\begin{example}
    The examples in Figure~\ref{fig:spreads} are all fp-spreads, with (ii) and (ii) being single-source. An example of a region which is convex but not connected can be obtained by removing the point $a$ from the shaded region of Figure~\ref{fig:spreads}(iii).
\end{example}

\begin{remark}\label{rem:cover}
We note that subsets of the form $\langle A,\infty\langle$ and $\langle A,B\langle$ are always convex, but will not necessarily be connected. The name fp-spread will be justified in Corollary~\ref{cor:fp_spread}.
\end{remark}

Finally, we recall that $\cP$ is a \emph{join-semilattice} (or \emph{upper semilattice}) if each pair of elements $x, y \in \cP$ admits a \emph{join}, or least upper bound, denoted $x \vee y$. It will often be convenient for us to restrict to posets which are either finite or join-semilattices.

\subsection{Persistence modules}

Fix a field $\field$. We consider $\cP$ as a category with set of objects $\cP$ and with a single morphism $(x,y) \in \Hom_{\cP}(x,y)$ for each segment $\langle x,y\rangle$. By either a \emph{($\cP$-)persistence module} or a \emph{representation} of $\cP$, we mean a (covariant) functor from $\cP$ to the category of $\field$-vector spaces. We denote by $\Rep \cP$ the category of $\cP$-persistence modules.

We also consider $\field_\cP$ the (generally non-unital) $\field$-algebra with vector-space basis the set of segments in $\cP$ and multiplication of basis elements given by $\langle x,y\rangle \cdot \langle z,w\rangle = \langle x,w\rangle$ if $z = y$ and $\langle x,y\rangle \cdot \langle z,w\rangle = 0$ otherwise.
We denote by $\mathrm{Mod}(\field_\cP)$ the category of right $\field_\cP$-modules. 
If $\cP$ is finite, then $\field_\cP$ is the ($\field$-)incidence algebra of $\cP$ in the sense of \cite{DRS}. In this case, it is well-known (see e.g. \cite[Chapter~5]{simson}) that the categories $\Rep \cP$ and $\mathrm{Mod}(\field_\cP)$ are equivalent, but this is still true without the finiteness assumption. Explicitly, there are inverse equivalences $\Phi: \Rep \cP \rightleftarrows \mathrm{Mod}(\field_\cP): \Psi$ as follows.
\begin{itemize}
    \item For $M \in \Rep \cP$, one has $\Phi(M) = \bigoplus_{x \in \cP} M(x)$ as an abelian group. The right action is then given by the linear extension of the rule $$\Pi_z(v \cdot \langle x,y\rangle) = \begin{cases} M(x,y)(\Pi_x(v)) & z = y\\ 0 & \text{otherwise},\end{cases}$$
    for $x, y, z \in \cP$ with $x \leq y$, $v \in \Phi(M)$, and $\Pi_z$ and $\Pi_x$ the projection maps $\Phi(M) \rightarrow M(z)$ and $\Phi(M) \rightarrow M(x)$.
   \item For $V \in \Rep \cP$, $x \leq y \in \cP$, and $v \in V$, one has
   $$
       \Psi(V)(x) = V \cdot \langle x,x\rangle,\qquad\qquad
       \Psi(V)(x,y)(v) = v \cdot \langle x,y\rangle.
   $$
\end{itemize}
In light of this fact, we identify $\Rep \cP$ and $\mathrm{Mod}(\field_\cP)$ for the remainder of this paper.

In practice, one often considers only \emph{pointwise finite-dimensional (pwf)} representations, which are those representations $M$ for which each $M(x)$ is finite-dimensional. Given a spread $\cS \subseteq \cP$, the indicator module $\mathbb{I}_\cS$, known as a \emph{spread module}, is a notable example. Explicitly, $\mathbb{I}_\cS$ is given by
$$\mathbb{I}_\cS(x) = \begin{cases} \field & x \in \cS\\0 & \text{otherwise,}\end{cases}\qquad \mathbb{I}_\cS(x,y) = \begin{cases} 1_\field & x, y \in \cS\\0 & \text{otherwise.}\end{cases}$$
We analogously refer to \emph{fp-spread modules}, \emph{upset modules}, etc.

A description of the Hom-space between spread modules over $\mathbb{R}^n$ was given in \cite[Proposition~14]{DX}. See also \cite[Proposition~3.10]{miller} (which considers upset modules over arbitrary posets) and \cite[Proposition~5.5]{BBH} (which considers spread modules over finite posets). The statement and proof of this result is readily generalised to spread modules over arbitrary posets, thus yielding the following.

\begin{proposition}\label{prop:hom_spread}
    Let $\cS,\cT \subseteq \cP$ be spreads. We denote by $\mathfrak{I}(\cS,\cT)$ the set of connected components $\mathcal{U}$ of $\cS \cap \cT$ which satisfy
    $$\{x \in \cS \mid x \leq \mathcal{U}\} \subseteq \mathcal{U} \text{ and } \{y \in \cT \mid \mathcal{U} \leq y\} \subseteq \mathcal{U}.$$
    Then $\dim_\field\Hom_\cP\left(\mathbb{I}_{\cS},\mathbb{I}_{\cT}\right) = |\mathfrak{I}(\cS,\cT)|$. Moreover, for all $\mathcal{U} \in \mathfrak{I}(\cS,\cT)$, there exists a morphism $f: \mathbb{I}_\cS \rightarrow \mathbb{I}_\cT$ with $\im(f) = \mathbb{I}_{\mathcal{U}}$.
\end{proposition}

In particular, each spread module has endomorphism ring isomorphic to $\field$, and is thus indecomposable. Moreover, if $\cP$ is totally ordered, then every indecomposable pwf representation is (isomorphic to) a spread module by \cite[Theorem~1.2]{BCB} (see also \cite[Section~3.6]{GR}). We will discuss the implications of this fact for 1-parameter persistence theory in Section~\ref{sec:invariant}.

While natural to consider, the category of pwf representations of an infinite poset can exhibit several pathological homological properties. For example, over $\mathbb{R}$ with the standard order, there exist representations which are projective\footnote{Recall that an object $P$ in an abelian category $\mathcal{A}$ is \emph{projective} if $\mathrm{Ext}^1_\mathcal{A}(P,-) = 0$, or equivalently if $\Hom_\mathcal{A}(P,-)$ is exact.} in the category of pwf representations but not projective in the category $\Rep \cP$, see e.g. \cite[Section~2.1]{IRT} and \cite[Theorem~1.2]{BM}.
To avoid such behavior, we restrict to the category $\rep \cP$ (resp. $\mods(\field_\cP)$) of \emph{finitely presentable}\footnote{Recall that an object $M$ in an abelian category $\mathcal{A}$ is \emph{finitely presentable} if it is isomorphic to the cokernel of a morphism between finitely generated projectives. In our setting, we take $\mathcal{A} = \Rep(\cP)$.} representations (resp. modules). 

\begin{remark}
    It is well-known (e.g. from the theory of finite-dimensional algebras as developed in \cite{ASS} or \cite{ARS}) that the categories of finitely presentable and pwf persistence modules coincide for finite posets. Over infinite posets, our decision to restrict to finitely presentable modules is consistent with the approach taken in many recent papers, e.g. \cite{BS,BOO,LW}. There are, however, examples of pwf persistence modules which are not finitely presentable but are still ``tame'', and thus have nice homological properties. We refer to \cite{miller} and \cite[Section~8]{CJT} for a more detailed discussion.
\end{remark}

\begin{remark} Projective representations are also called \emph{free persistence modules}, especially in the context where $\Rep(\mathbb{R}^n)$ is viewed as the category of $\mathbb{Z}^n$-graded modules over the polynomial ring in $n$ variables.\end{remark}

The following is well-known, but we give an outline of a proof for convenience.

\begin{theorem}\label{thm:projective}
    Let $M \in \Rep \cP$ be indecomposable. Then $M$ is projective if and only if there exists $x \in \cP$ such that $M \cong \mathbb{I}_{\langle x,\infty\langle}$.
\end{theorem}

\begin{proof}
    There is a natural isomorphism $\Hom_\cP(\mathbb{I}_{\langle x,\infty\langle},-) \rightarrow (-)(x)$ given by sending $h: \mathbb{I}_{\langle x,\infty\langle} \rightarrow N$ to $h_x(1_\field)$. This implies that $\Hom_\cP(\mathbb{I}_{\langle x,\infty\langle},-)$ is exact and so $\mathbb{I}_{\langle x,\infty\langle}$ is projective. Moreover, it implies that every $N \in \Rep(\cP)$ is isomorphic to a quotient of a direct sum of representations of the form $\mathbb{I}_{\langle x,\infty\langle}$. Since the endomorphism ring of each $\mathbb{I}_{\langle x,\infty\langle}$ is local, the Krull-Remak-Schmidt-Azumaya Theorem \cite{azumaya} (see also \cite[page~17-12]{gabriel}) thus implies that there are no other indecomposable projectives.
\end{proof}

\begin{remark}
    $\mathbb{I}_{\langle x,\infty\langle} = \langle x,x\rangle \cdot \field_\cP$ as a right $\field_\cP$-module.
\end{remark}

\begin{corollary}\label{cor:fp_spread}
    \begin{enumerate}
        \item Every finitely presentable persistence module is pwf.
        \item Suppose that $\cP$ is finite or a join-semilattice, and let $\cS \subseteq \cP$ be a spread. Then $\mathbb{I}_{\cS}$ is finitely presentable if and only if $\cS$ is an fp-spread.
    \end{enumerate}
\end{corollary}

\begin{proof}
    (1) is an immediate consequence of the fact that $\mathbb{I}_{\langle x,\infty\langle}$ is pwf for all $x \in \cP$.

    (2) If $\cP$ is finite then (a) every spread is both pwf and an fp-spread, and (b) every pwf representation is finitely presentable. Thus suppose $\cP$ is a join-semilattice. Let $A \subseteq \cP$ be a finite antichain and consider the multiset $A' = \{x \vee y \mid x \neq y \in A\}$. One can then use Proposition~\ref{prop:hom_spread} to show that there is an exact sequence
    $\bigoplus_{z \in A'} \mathbb{I}_{\langle z,\infty\langle} \rightarrow \bigoplus_{x \in A} \mathbb{I}_{\langle x,\infty\langle} \rightarrow 0.$
    This shows that $\mathbb{I}_{\langle A,\infty\langle} \in \rep \cP$. Moreover, for $B \subseteq \cP$ a finite antichain with $B \subseteq \langle A,\infty\langle$, there is a short exact sequence
    $0 \rightarrow \mathbb{I}_{\langle B,\infty\langle} \rightarrow \mathbb{I}_{\langle A,\infty\langle} \rightarrow \mathbb{I}_{\langle A,B\langle} \rightarrow 0,$
    again by Proposition~\ref{prop:hom_spread}.
    Thus $\mathbb{I}_{\langle A,B\langle} \in \rep \cP$ since $\rep \cP$ is abelian.

    Conversely, let $\cS \subseteq \cP$ be a spread and suppose that $\mathbb{I}_\cS \in \rep \cP$. Thus there exist finitely generated projectives $P, Q \in \rep \cP$ and a morphism $r: P \rightarrow Q$ such that $\mathbb{I}_\cS \cong \coker r$. It follows that for all $x \in \cS$ there exists $y \leq x \in \cS$ such that $\mathbb{I}_{\langle y,\infty\langle}$ is a direct summand of $Q$. Since $Q$ is finitely generated, this means there exists a finite antichain $A \subseteq \cS$ such that every $x \in \cS$ satisfies $A \leq x$. If $\cS$ is an upset, this implies that it is an fp-spread. Otherwise, let $\mathcal{T} = \{z \in \cP \mid \cS \lneq z\}$, and note that we can write $\mathcal{T} = \bigsqcup_{\mathcal{U} \in \mathfrak{T}}\mathcal{U}$ as a disjoint union of connected upsets. Proposition~\ref{prop:hom_spread} can then be used to construct a short exact sequence
    $0 \rightarrow \bigoplus_{\mathcal{U} \in \mathfrak{T}} \mathbb{I}_{\mathcal{U}} \rightarrow \mathbb{I}_{\langle A,\infty\langle} \rightarrow \mathbb{I}_\cS \rightarrow 0.$
    Since $\rep \cP$ is abelian, the same argument used for $\cS$ shows that $\mathfrak{T}$ is finite and that each $\mathcal{U}$ is of the form $\langle B_\mathcal{U},\infty\langle$ for some $B_\mathcal{U} \subseteq \cP$ a finite antichain. Taking $B = \bigsqcup_{\mathcal{U} \in \mathfrak{T}} B_\mathcal{U}$, we conclude that $\cS = \langle A,B\langle$ is an fp-spread.
\end{proof}

\begin{example}
    Suppose $\cP = \{x,y\} \sqcup \mathbb{Z}$ with relations $x \leq i$ and $y \leq i$ for all $i \in \mathbb{Z}$. Then $\cP$ is not a join-semilattice and $\mathbb{I}_{\langle \{x,y\},\infty\langle}$ is not finitely presentable.
\end{example}

\begin{remark}
    One consequence of Corollary~\ref{cor:fp_spread} is that the category $\rep \cP$ is \emph{Krull-Remak-Schmidt}, which informally means every representation can be uniquely (in some precise sense) expressed as a finite direct sum of indecomposable subrepresentations. Indeed, Corollary~\ref{cor:fp_spread}(1) and the main result of \cite{BCB} imply that every $M \in \rep \cP$ admits a (potentially infinite) direct sum decomposition which satisfies the necessary uniqueness property. The finiteness of the decomposition is then a consequence of the charactersablion of indecomposable projectives given in Theorem~\ref{thm:projective}.
\end{remark}

\subsection{Invariants}\label{sec:invariant}

We now build towards a formal definition of an invariant and discuss several examples. Note that, while we are working only with finitely presentable persistence modules, the general theory of invariants can be formulated in any additive category. See also Remark~\ref{rem:pwf} for an explicit discussion of the pwf setting.

Since $\rep \cP$ is skeletally small, we adopt the common convention of collapsing isomorphism classes. That is, when we say ``the set of objects in a subcategory'' we assume that we have fixed a single representative for each isomorphism class.

Recall that a subcategory $\mathcal{C} \subseteq \rep \cP$ is \emph{additive} if it is closed under finite direct sums and direct summands. Thus an additive category is uniquely determined by its set of indecomposable objects. For $\mathcal{C} \subseteq \rep \cP$, we denote by $\add(\mathcal{C})$ the smallest additive category containing $\mathcal{C}$, whose objects are all finite direct sums of direct summands of objects in $\mathcal{C}$. We also denote by $\ind(\mathcal{C})$ the set of indecomposable objects in $\mathcal{C}$.

\begin{definition}\label{def:invariant}
    Let $\mathcal{C} \subseteq \rep \cP$ be an additive subcategory and denote by $K^{\mathrm{split}}_0(\mathcal{C})$ the free abelian group with basis $\ind(\mathcal C)$\footnote{This is the \emph{split Grothendieck group} of $\mathcal{C}$. See Definition~\ref{def:grothendieck} for additional discussion.}.
    \begin{enumerate}
        \item An \emph{additive invariant} (from here on referred to just as an \emph{invariant}) on $\mathcal{C}$ is a group homomorphism $p: K^{\mathrm{split}}_0(\mathcal{C}) \rightarrow G$ for some abelian group $G$. Alternatively, $p$ is an association $\mathcal{C} \rightarrow G$ which is constant on isomorphism classes and satisfies $p(M \oplus N) = p(M) + p(N)$.
        \item Let $p: K^{\mathrm{split}}_0(\mathcal{C}) \rightarrow G$ and $p': K^{\mathrm{split}}_0(\mathcal{C}) \rightarrow G'$ be two additive invariants. We say that $p$ is \emph{equivalent} to $p'$ if $\ker(p) = \ker(p')$ and that $p$ is \emph{finer} than $p'$ if $\ker(p) \subseteq \ker(p')$.
        \item We say an additive invariant $p: K^{\mathrm{split}}_0(\mathcal{C}) \rightarrow G$ is \emph{complete} on $\mathcal{C}$ if $\ker(p) = 0$. Alternatively, $p$ is equivalent to the identity on $K^{\mathrm{split}}_0(\mathcal{C})$, and thus $p(M) = p(N)$ if and only if $M \cong N$.
    \end{enumerate}
\end{definition}

We will mostly be interested in the case $\mathcal{C} = \rep \cP$ in Definition~\ref{def:invariant}.

\begin{example}\label{ex:invariants}
    \begin{enumerate}
        \item The first example of an invariant is the \emph{dimension vector} or \emph{Hilbert function}. This associates to each $M \in \rep(\cP)$ the vector 
        $$\underline{\dim}(M) = (\dim_\field M(x))_{x \in \cP} \in \mathbb{Z}^{\cP}.$$
        If e.g. $\cP$ is finite or is a product of finitely many totally ordered sets, then $\rep(\cP)$ is well-known to have finite global dimension. The dimension vector is then complete on the category of projective representations.
        \item Suppose $\cP$ is totally ordered. Then the \emph{barcode} associates $M \in \rep\cP$ to the vector of multiplicities of each $\mathbb{I}_{\langle x,y\langle}$ as a direct summand of $M$. As previously mentioned, the spread modules are the only indecomposables in this case and the category $\rep(\cP)$ is Krull-Remak-Schmidt. Thus the barcode is complete on $\rep \cP$.
        \item The \emph{rank invariant} was the first invariant proposed for multiparameter persistence modules in \cite{CZ}, and a similar concept also appeared independent of persistence theory in \cite{kinser}. This associates each $M \in \rep(\cP)$ to the vector $$\underline{\mathrm{rk}}(M) = (\mathrm{rank} \ M(x,y))_{x \leq y \in \cP} \in \mathbb{Z}^\mathcal{T},$$
        where $\mathcal{T} = \{(x,y) \in \cP^2 \mid x \leq y\}$. This invariant is always finer that the dimension vector, and is equivalent to the barcode when $\cP$ is totally ordered. The work of \cite{BOO} implies that the rank invariant is complete on the category $\add(\cX)$ for $\cX$ the union of the projectives and the hook modules.
    \end{enumerate}
\end{example}

\begin{remark}\label{rem:pwf}
    Each of the examples described in Example~\ref{ex:invariants}, along with the stated completeness results, can be extended to the category of pwf persistence modules. See \cite{CB,BCB,IRT} for (1) and (2) and \cite{BOO} for (3).
\end{remark}

Another source of invariants comes from recording the dimension of Hom-spaces.

\begin{definition}\cite[Definition~4.12]{BBH}\label{def:dim_hom}
     Let $\cX$ be a set of indecomposable objects in $\rep \cP$ and let $\mathcal{C} \subseteq \rep \cP$ be an additive subcategory. We say that an invariant $p: K_0^{\mathrm{split}}(\mathcal{C}) \rightarrow G$ is a \emph{dim-hom invariant} relative to $\cX$ if it is equivalent to the invariant $K_0^{\mathrm{split}}(\mathcal{C}) \rightarrow \prod_{X \in \cX} \mathbb{Z}\cdot [X]$ given by $M \mapsto \left((\dim_\field\Hom_\cP(X,M))\cdot [X]\right)_{X \in \cX}.$
\end{definition}

\begin{example}
    The dimension vector is classically known to be an example of a dim-hom invariant relative to $\cX$ the set of indecomposable projectives. Similarly, if $\cP$ is totally ordered, then the barcode is a dim-hom invariant relative to $\cX$ the set of fp-spread modules (see e.g. \cite[Section~7.1]{BBH}, which explains this in the finite case). Finally, the rank invariant is shown to be a dim-hom invariant in \cite[Section~4]{BOO} (see also \cite[Theorem~7.2]{BBH}).
\end{example}

As defined, many of the invariants we have discussed are difficult to encode over infinite posets. Indeed, if $\cP$ is infinite, then the dimension vector and rank invariant both take values in a group isomorphic to an infinite product of copies of $\mathbb{Z}$. One approach taken in many recent papers (e.g. \cite{AENY,AENY2,BBE,BOO,BOOS,BE,KM,MP,OS,patel,thomas_thesis}) is to utilize tools such as M\"obius inversion and Grothendieck groups to form a ``signed approximation'' or ``generalized persistence diagram''. Informally, this process ``approximates'' a given persistence module by a \emph{pair} of modules in some class in a way that respects a given invariant. To be precise, we consider the following definition motivated by the works \cite{BOO,BOOS,OS}.

\begin{definition}\label{def:signed_decomposition}
    Let $\mathcal{C} \subseteq \rep \cP$ be an additive subcategory and $p: K_0^{\mathrm{split}}(\mathcal{C}) \rightarrow G$ an invariant. Let $\cX \subseteq \ind(\mathcal{C})$. For $M \in \mathcal{C}$, we say that a pair $(M_\cX^0,M_\cX^1)$ is a \emph{signed $\add(\cX)$-decomposition of $p$} if $M_\cX^0 \oplus M_\cX^1 \in \add(\cX)$ and $p(M) = p(M_{\cX}^0) - p(M_\cX^1)$. We further say that $(M_\cX^0,M_\cX^1)$ is a \emph{minimal signed $\add(\cX)$-decomposition of $p$} if $\add(M_\cX^0) \cap \add(M_\cX^1) = 0$.
\end{definition}

Given a canonical choice of signed decomposition $(M_\cX^0,M_\cX^1)$ for each $M \in \mathcal{C}$, one automatically obtains an invariant by recording the multiplicities of each $X \in \mathcal{X}$ as a direct summand of $M_\cX^0$ and/or $M_\cX^1$. Furthermore, it has been shown in \cite[Section~8.3]{OS} and \cite[Sections~6 and~7]{BOOS} that signed decompositions are useful in establishing stability results for invariants coming from (relative) homological algebra. We discuss signed decompositions further in section~\ref{sec:homological}.


\section{General theory of exact structures}\label{sec:exact}

In this section, we recall several definitions and results related to exact structures\footnote{See Remark~\ref{rem:structure_vs_category} for a comparison with the notion of an exact \emph{category}.} on categories of persistence modules (and more generally on abelian categories). We note that the definition of an exact structure is more generally formulated for additive categories, but we restrict our attention to the special case of abelian categories in order to streamline the exposition. We refer to \cite{BBHG,buhler,DRSS,EJ} and the references therein for additional background information and historical context.

\subsection{Exact structures}

Fix an abelian category $\cA$. We write monomorphisms in $\cA$ as $X \mono{f} Y$ (or just $X \mono{} Y$ when the morphism $f$ is implied or unimportant) and epimorphisms in $\cA$ as $Y \epi{g} Z$ (or just $Y \epi{} Z$ when the morphism $g$ is implied or unimportant). We also adopt the following standard terminology.

\begin{definition}\label{def:admissible}
    Let $\cE$ be a class of short exact sequences in $\cA$ which is closed under isomorphisms.
    \begin{enumerate}
        \item We say a short exact sequence $\eta$ in $\cA$ is \emph{$\cE$-exact} if $\eta \in \cE$.
        \item Let $\eta = (\cdots \xrightarrow{f_0} X_1 \xrightarrow{f_1} X_2 \xrightarrow{f_2} \cdots)$ be a long exact sequence in $\cA$. We say that $\eta$ is $\cE$-exact if the sequences $\ker f_i \mono{} X \epi{} \im f_i$ are $\cE$-exact for all $i$.
        \item Let $A \xrightarrow{h} B$ be a morphism in $\cA$. We say that $h$ is \emph{$\cE$-admissible} if $\ker h \mono{} A \epi{} \im h$ and $\im h \mono{} B \epi{} \coker h$ are both $\cE$-exact.
    \end{enumerate}
\end{definition}
 
 We now recall the definition of an exact structure. We use the set of axioms found in \cite[Section~2]{buhler}, which are shown in \cite{keller} to be equivalent to the original axioms found in Quillen's seminal work \cite{quillen}.

 \begin{definition}\label{def:exact}
Let $\cE$ be a class of short exact sequences in $\cA$ which is closed under isomorphisms. Then $\cE$ is said to be an \emph{exact structure} on $\mathcal{A}$ if all of the following hold:
 \begin{enumerate}
	\item For all $X \in \cA$ the identity $1_X$ is $\cE$-admissible.
        \item If $X \mono{f} Y$ and $Y \mono{g} Z$ are both $\cE$-admissible, then $X \mono{g \circ f} Z$ is $\cE$-admissible.
        \item If $X \epi{f} Y$ and $Y \epi{g} Z$ are both $\cE$-admissible, then $X\epi{g\circ f} Z$ is $\cE$-admissible.
        \item If $X \mono{f} Y$ is $\cE$-admissible and $X\xrightarrow{h} W$ is any morphism in $\cA$, then the pushout $f_h$ of $f$ along $h$ is $\cE$-admissible (and a monomorphism since $\cA$ is abelian).
        \item If $Y \epi{g} Z$ is $\cE$-admissible and $W \xrightarrow{h}Z$ is any morphism in $\cA$, then the pullback $g^h$ of $g$ along $h$ is $\cE$-admissible (and an epimorphism since $\cA$ is abelian).
	\end{enumerate}
\end{definition}

\begin{example}
    When viewed as classes of short exact sequences, the exact structures on $\cA$ can be ordered by containment \cite{BHLR}. The minimal exact structure with respect to this relation contains only the split short exact sequences, and is denoted $\cE_{\mathrm{split}}$. The maximal exact structure with respect to this relation contains all short exact sequences, and is denoted $\cE_{\mathrm{all}}$. This is also called the \emph{standard} exact structure.
\end{example}

\begin{remark}\label{rem:structure_vs_category}
    There are two conventions of terminology that are both used in the literature. One is to refer to $\cE$ as an exact structure on $\cA$ (as we have done here), and the other is to refer to the pair $(\cA,\cE)$ as an \emph{exact category}. We use the term exact structure throughout the present paper since we will often consider multiple exact structures on the same abelian category.
\end{remark}

\begin{remark}\label{rem:admissible}
    It is an immediate consequence of the definition
    that any of the following classes uniquely characterizes an exact structure $\cE$.
    \begin{enumerate}
        \item The class of short $\cE$-exact sequences.
        \item The class of $\cE$-admissible monomorphisms.
        \item The class of $\cE$-admissible epimorphisms.
    \end{enumerate}
\end{remark}

\begin{remark}\label{rem:closed}
    Alternatively, one can see an exact structure as a ``closed'' additive subbifunctor of $\mathrm{Ext}^1_{\cA}(-,-)$. See e.g. \cite[Theorem~1.1]{BH} and \cite[Corollary~1.6]{DRSS} for additional details on this perspective.
\end{remark}

We now recall a family of exact structures defined in \cite[Section~1]{AS} for finite-dimensional algebras and in \cite[Section~1.3]{DRSS} for arbitrary abelian categories.

\begin{definition}\label{def:AS}
    Let $\cX$ be a class of objects in $\cA$.
    \begin{enumerate}
        \item We denote by $\cF_\cX$ the class of short exact sequences $\eta = (L \mono{f} M \epi{g} N)$ in $\cA$ for which the induced sequences
    $$0 \rightarrow \Hom_\cA(X,L) \xrightarrow{f_*} \Hom_\cA(X,M) \xrightarrow{g_*} \Hom_\cA(X,N) \rightarrow 0$$
    are exact for all $X \in \cX$. Equivalently, $\eta \in \cF_\cX$ if and only if $g_* = \Hom_\cA(X,g)$ is an epimorphism for all $X \in \cX$.
         \item We denote by $\cF^\cX$ the class of short exact sequences $\eta = (L \mono{f} M \epi{g} N)$ in $\cA$ for which the induced sequences
    $$0 \rightarrow \Hom_\cA(N,X) \xrightarrow{g^*} \Hom_\cA(M,X) \xrightarrow{f^*} \Hom_\cA(L,X) \rightarrow 0$$
    are exact for all $X \in \cX$. Equivalently, $\eta \in \cF^\cX$ if and only if $f^* = \Hom_\cA(f,X)$ is a monomorphism for all $X \in \cX$.
    \end{enumerate}
\end{definition}

\begin{remark}
    With the exception of Example~\ref{ex:spreads_infinite}, we work only with classes of the form $\cF_\cX$, as opposed to $\cF^\cX$. We note, however, that the results presentable for $\cF_\cX$ also hold in their dual forms for $\cF^\cX$.
\end{remark}

\begin{remark}
    We note that if $\cX$ is additive then $\cF_\cX = \cF_{\ind \cX}$. It is therefore common practice to replace $\cX$ with $\ind(\cX)$ when applying Definition~\ref{def:AS}.
\end{remark}

We will typically apply Definition~\ref{def:AS} only when all of the objects in $\cX$ are indecomposable, but it will be useful to use the more general formulation in Section~\ref{sec:infinite_exact_structure}.

\begin{proposition}\label{prop:exact_gen}\cite[Proposition~1.7]{DRSS}\cite[Proposition~1.7]{AS}
    Let $\cX$ be a class of objects in $\cA$. Then $\cF_\cX$ is an exact structure.
\end{proposition}

\begin{remark}
   The references \cite{AS,DRSS} prove Proposition~\ref{prop:exact_gen} using the language of ``closed'' additive subfunctors, but it is a straightforward excercise to prove directly that the axioms are satisfied. We also note that \cite[Proposition~1.7]{AS} is stated only for $\cA = \mods\Lambda$ with $\Lambda$ an artin algebra, but, as observed in \cite[Section~1.3]{DRSS}, the statement and its proof are both valid in any abelian category.
\end{remark}

\subsection{Relative projective objects}

We will generally characterize exact structures in terms of their corresponding relative projective objects, defined as follows.

\begin{definition}\label{def:projectives}
	Let $\cE$ be an exact structure on $\cA$. We say an object $U \in \cA$ is \emph{$\cE$-projective}, in symbols $U \in \proj(\cE)$, if $g_* = \Hom_\cA(U,g)$ is an epimorphism for all $\cE$-admissible epimorphisms $Y \epi{g} Z$. Equivalently, $U \in \proj(\cE)$ if and only if $\Hom_\cA(U,-)$ sends $\cE$-exact sequences to exact sequences.
\end{definition}

\begin{notation}
    As a special case, we denote $\proj(\cA) = 
    \proj(\cE_{\max})$. This is the class of objects which are projective in $\cA$ is the classical sense.
\end{notation}

The following is a straightforward consequence of the definitions.

\begin{lemma}\label{lem:F_proj_included}
    Let $\cX$ be a class of objects in $\cA$. Then $\add(\cX \cup \proj(\cA)) \subseteq \proj(\cF_\cX)$.
\end{lemma}

As a partial converse to Lemma~\ref{lem:F_proj_included}, we have the following.

\begin{theorem}\label{thm:AS_1}\cite[Proposition~1.10]{AS}
    Let $\Lambda$ be an artin algebra and $\cA = \mods \Lambda$, and let $\cX$ be a class of objects in $\cA$. Then $\add(\cX \cup \proj(\cA)) = \proj(\cF_\cX)$.
\end{theorem}

\begin{remark}\label{rem:AR}
    The proof of \cite[Proposition~1.10]{AS} is based on the theory of \emph{almost split sequences} (also known as Auslander-Reiten sequences). Indeed, it is shown in \cite[Proposition~1.9]{AS} that an indecomposable nonprojective (in the clasical sense) module $U \in \mods \Lambda$ is $\cF_\cX$-projective if and only if the almost split sequence ending in $U$ is not $\cF_\cX$-exact. As discussed in \cite[Section~2]{DRSS}, this result thus generalizes to any Krull-Remak-Schmidt category which ``has almost split pairs''. There are, however, many  interesting examples of abelian categories which do not have almost split pairs, including the category $\rep(\mathbb{R}^n)$\footnote{There is some subtlety here. Indeed, the work of Auslander \cite{AR_fp} says that, for any $\cP$, every object in $\rep \cP$ is the third term of an almost split sequence in $\Rep \cP$, but there is no requirement that the other terms of the sequence be finitely presentable.}.
\end{remark}

For emphasis, we record the following.

\begin{corollary}\label{cor:AS_1}\cite[Proposition~1.10]{AS}
    Let $\cP$ be a finite poset and let $\cX$ be a class of objects in $\rep \cP$. Then $\add(\cX \cup \proj(\rep \cP)) = \proj(\cF_\cX)$.
\end{corollary}

It is not immediately clear whether Corollary~\ref{cor:AS_1} holds over arbitrary infinite posets, since the category $\rep \cP$ does not generally have almost split pairs in this case (see the footnote to Remark~\ref{rem:AR} for a clarification of this claim). We will, however, show in Section~\ref{sec:infinite_exact_structure} that the result holds for certain choices of $\cX$ in the case where $\cP$ is a finite product of totally ordered sets.


\section{Homological invariants}\label{sec:approximations}

In this section, we recall the notions of covers, resolutions, and Grothendieck groups relative to an exact structure. We follow much of the exposition of \cite[Section~4.1]{BBH}, but we work only in the category $\rep \cP$ for $\cP$ an arbitrary poset\footnote{Note that \cite[Section~4.1]{BBH} is written for the category $\mods \Lambda$ for $\Lambda$ a finite-dimensional algebra over a field. Thus our setting is neither strictly more nor strictly less general.}. See also \cite[Section~2]{CGRST} for a detailed treatise over finite posets which does not rely on the theory of finite-dimensional algebras. We also fix a class $\cX \subseteq \rep \cP$ such that $\proj(\rep \cP) \subseteq \add(\cX)$.

\subsection{Approximations and resolutions}

\begin{definition}\label{def:approx}
     Let $M \in \rep \cP$. An \emph{ $\add(\cX)$-precover} (also known as a \emph{ right $\add(\cX)$-approx-\linebreak imation}) of $M$ is an object $U \in \add(\cX)$ and an $\cF_\cX$-admissible epimorphim $f: U \rightarrow M$. We say that $f$ is an \emph{ $\add(\cX)$-cover} (also known as a \emph{minimal right $\add(\cX)$-approximation} of $M$) if in addition every endomorphism $g: U \rightarrow U$ such that $f = f \circ g$ is an automorphism. We say that $\add(\cX)$ is \emph{precovering}
    if every object of $\rep \cP$ admits an $\add(\cX)$-precover. If in addition $\add(\cX) = \proj(\cF_\cX)$, we say that $\cF_\cX$ \emph{has enough projectives.}
\end{definition}

\begin{remark}
    The category $\rep(\cP)$ has $\add(\cX)$-precovers for any finite set $\cX$. This is the only setting considered in \cite{BBH}.
\end{remark}

Since $\rep \cP$ is Krull-Remak-Schmidt, the following is well-known.

\begin{lemma}\label{lem:cover}
    Let $M \in \rep(\cP)$. Then $M$ admits an $\add(\cX)$-cover if and only if $M$ admits an $\add(\cX)$-precover.
    Moreover, if $f: U \rightarrow M$ and $g: V \rightarrow M$ are both $\add(\cX)$-covers of $M$, then there is an isomorphism $h: U \rightarrow V$ such that $f = g \circ h$.
\end{lemma}

\begin{proof}[Sketch of Proof]
    For $f: U \rightarrow M$ an $\add(\cX)$-precover and $h: U \rightarrow U$ a non-isomorphism with $f = f\circ h$, Fitting's Lemma yields a decomposition $U = \mathrm{im}(f^n) \oplus \ker(f^n)$ for some $n$. (See e.g. \cite[Lemmas~5.2 and~5.3]{krause_proj} for a proof that Fitting's Lemma applies.) Then $h: \mathrm{im}(f^n) \rightarrow M$ is an $\add(\cX)$-precover with strictly fewer indecomposable direct summands than $U$, meaning that an $\add(\cX)$-cover must exist. Now given two covers $f$ and $g$, the definition of a precover yields morphisms $h: U \rightarrow V$ and $h': V \rightarrow U$ with $f = g\circ h$ and $g = f\circ h'$. Then $h\circ h'$ and $h' \circ h$ are both isomorphisms by the definition of a cover.
\end{proof}

By definition, the standard exact structure $\cE_{\mathrm{all}} = \cF_{\proj(\rep \cP)}$ on $\rep \cP$ has enough projectives. In this case, any epimorphism from an object in $\proj(\rep \cP)$ to $M$ is a $\proj(\rep \cP)$-precover. For other choices of $\cX$, however, there may be epimorphisms from $\add(\cX)$ to $M$ which are not $\cF_\cX$ admissible, see e.g. \cite[Example~4.3(2)]{BBH}.

We now recall the following definitions, each of which recovers the corresponding classical notion from homological algebra if $\add(\cX) = \proj(\rep \cP)$.

\begin{definition}\label{def:resolution}
Suppose that $\add(\cX)$ is precovering and let $M \in \rep \cP$.
\begin{enumerate}
    \item An \emph{$\add(\cX)$-resolution} of $M$ is a $\cF_\cX$-exact sequence of the form
	$U_\bullet = (\cdots \xrightarrow{q_2} U_1 \xrightarrow{q_1} U_0)$
	such that $M = \coker(q_1)$ and each $U_i \in \add(\cX)$.
	We say $U_\bullet$ is a \emph{minimal} $\add(\cX)$-resolution if all of the $q_i$ and the quotient map $U_0 \rightarrow M$ are $\add(\cX)$-covers of their images. We say $U_\bullet$ is \emph{finite} if there exists some $j$ such that $U_{i} = 0$ for all $i > j$.
		\item Let $U_\bullet$ be a minimal $\add(\cX)$-resolution of $M$. Then the \emph{$\add(\cX)$-dimension} of $M$ is $\dim_{\cX}(M) = \min(\{j \mid U_{j+1} = 0\} \cup \{\infty\})$.
		\item The \emph{$\add(\cX)$-global dimension} of $\field_\cP$ is $\dim_\cX(\field_\cP) = \sup_{M \in \rep \cP} \dim_{\cX}(M).$ Furthermore, we say that the $\add(\cX)$-global dimension of $\field_\cP$ is \emph{realisably-infinite}\footnote{The term ``properly infinite'' is used in \cite{BBH} to describe this property.} if there exists $M \in \rep \cP$ such that $\dim_{\cX}(M) = \infty$.
	\end{enumerate}
\end{definition}

\begin{remark}\label{rem:unique}
   Suppose that $\add(\cX)$ is precovering. It follows from Lemma~\ref{lem:cover} that the minimal $\add(\cX)$-resolution of a given $M \in \rep \cP$ is unique up to isomorphism.
\end{remark}

\begin{remark}\label{rem:global}
    Note that if $\cP$ is finite, then the $\add(\cX)$-global dimension of $\field_\cP$ is infinite if and only if it is realisably infinite. See Remark~\ref{rem:generator} below.
\end{remark}

\begin{remark}
	In Definition~\ref{def:resolution}(1), the quotient map $U_0 \ \epi{} \coker(q_1) = M$ will automatically be $\cF_\cX$-admissible by Definition~\ref{def:admissible}(2).
\end{remark}

\begin{definition}\label{def:grothendieck}
	Let $F$ be the free abelian group generated by the symbols $[M]$ for every isomorphism class of $M \in \rep \cP$. We denote by $H_\cX$ the subgroup generated by
	$$\{[M]-[N]+[L] \mid M \mono{\phantom{f}} N \epi{\phantom{g}} L \text{ is $\cF_\cX$-exact}\}.$$
	The quotient $K_0(\cP,\cX)= F/H_\cX$ is called the \emph{Grothendieck group} of $\cP$ relative to $\cX$. Given $M \in \rep \cP$, we denote by $[M]_\cX$ (or just $[M]$ when $\cX$ is clear from context) the representative of $M$ in the Grothendieck group $K_0(\cP,\cX)$.
\end{definition}

The proof of \cite[Proposition~4.9]{BBH} readily generalizes to a proof of the following.

\begin{proposition}\label{prop:grothendieck}
	Suppose that $\add(\cX)$ is precovering and that the $\add(\cX)$-global dimension of $\field_\cP$ is not realisably infinite. Then:
	\begin{enumerate}
		\item The group $K_0(\cP,\cX)$ is free abelian with basis $\{[U] \mid U \in \ind(\add \cX)\}$.
		\item For any $M \in \rep \cP$ and any finite $\add(\cX)$-resolution $U_\bullet$ of $M$, one has $$[M]_\cX = \sum_{i = 0}^\infty (-1)^i[U_i]_\cX.$$
	\end{enumerate}
\end{proposition}


\subsection{Homological invariants}\label{sec:homological}

The following extends \cite[Definition~4.12(2)]{BBH} to include infinite classes of objects.

\begin{definition}
    Suppose that $\add(\cX)$ is precovering and that the $\add(\cX)$-dimension of $\rep \cP$ is not realisably infinite.
    Then we say any invariant equivalent to the association $M \mapsto [M]_\cX$ is a \emph{homological invariant} (relative to $\cX$).
\end{definition}

\begin{remark}\label{rem:Betti}
    The association $M \mapsto [M]_\cX$ can be seen as a way of ``collapsing'' a minimal $\add(\cX)$-resolution of $M$ into an invariant. There are multiple other ways that one could do this, most notably by considering the \emph{relative Betti numbers}. This approach records the direct sum decomposition of each term of a \emph{minimal} $\add(\cX)$-resolution and keeps them separate rather than forming an alternating sum (as is done in Proposition~\ref{prop:exact_gen}(2)). We refer to \cite{CGRST} for additional discussion of this approach.
\end{remark}

\begin{example}\
    \begin{enumerate}
        \item Since $\rep(\cP)$ has finite global dimension, the dimension vector is a homological invariant relative to $\cX = \ind(\proj(\rep \cP))$.
        \item It is shown in \cite[Section~4]{BOO} that the rank invariant is a homological invariant relative to $\cX$ the union of $\ind(\proj(\rep \cP))$ with the set of hook modules. Moreover, it is shown that the exact structure $\cF_\cX$ consists precisely of those exact sequences over which the association $M \mapsto \underline{\mathrm{rk}}(M)$ is additive.
        \item If $\cP$ is finite and either (i) $\cX$ contains only single-source spread modules, or (ii) $\cX$ contains only upset modules, then the association $M \mapsto [M]_\cX$ is a homological invariant relative to $\cX$. See also \cite[Theorem~6.12]{miller}, which considers $\cX$ the set of all upset modules over arbitrary posets and a larger category than $\rep(\cP)$.
        \item Let $\cP = \{1,\ldots,n\} \times \{1,\ldots,m\}$ be a product of two finite totally ordered sets and let $\cX$ be the set of spread modules. Then the association $M \mapsto [M]_\cX$ is a homological invariant relative to $\cX$ by \cite[Proposition~4.5]{AENY}.
    \end{enumerate}
\end{example}

We conclude with a brief discussion of signed decompositions (recall Definition~\ref{def:signed_decomposition}) of homological invariants. The following is an immediate consequence of Proposition~\ref{prop:grothendieck}(1).

\begin{corollary}\label{cor:signed_decomposition}
    Suppose $p: K_0^{\mathrm{split}}(\rep(\cP)) \rightarrow G$ is a homologcial invariant relative to $\cX$. Then $p$ is complete on $\add(\cX)$. Moreover, for $M \in \rep(\cP)$, write $[M]_\cX = \sum_{X \in \cX} m_X\cdot [X]_\cX$ and denote $$M_\cX^+ = \bigoplus_{X \in \cX} X^{\max\{0,m_X\}}, \qquad\qquad M_\cX^- = \bigoplus_{X \in \cX} X^{\max\{0,-m_X\}}.$$
    Then $\left(M_\cX^+,M_\cX^-\right)$ is the unique minimal signed $(\add \cX)$-decomposition of $p$.
\end{corollary}

\begin{remark}
    While the above corollary establishes the uniqueness of and gives a recipe for computing minimal signed decompositions of homological invariants, it is still useful to consider other signed decompositions. Indeed, in \cite[Section~8.3]{OS} and \cite[Sections~6 and~7]{BOOS}, it is necessary to use signed decompositions which are not minimal, but can still be computed from minimal $\add(\cX)$-resolutions (and thus deduced from the relative Betti numbers).
\end{remark}


\section{Global dimensions and the representation dimension}\label{sec:rep_dim}

We consider in this section an additively finite class 
$\cX$, that is $\cX = \add \; G$ for some $\cP$-module $G$. Keeping the condition $\proj(\rep \cP) \subseteq \add(\cX)$, this means precisely that we fix a {\em generator $G$ of $\rep \cP$}, with $\cP$ necessarily finite here since $G$ contains all indecomposable projectives.
If in addition $G$ contains all indecomposable injective modules, then $G$ is called a {\em generator-cogenerator}.
These are the building blocks for Auslander's notion of representation-dimension, defined in \cite{Aus} for any artin algebra, which we study here for posets:

\begin{definition}\label{def:repdim}
The \emph{representation dimension} of $\cP$ is 
\[ \repdim \field_\cP =  \inf \{\gldim(\End_{\cP}(G)) \; | \; G \mbox{ is a generator-cogenerator of } \rep \cP \}.\]
\end{definition}

Before giving an example, we recall some historical context. Auslander introduced the notion of representation dimension as a means of deciding the representation type of an artin algebra $\Lambda$, showing that $\Lambda$ is representation-finite precisely when $\repdim \Lambda \le 2$. 
Here $\Lambda$ is said to be \emph{representation-finite} when the number of indecomposable $\Lambda$-modules is finite, and in this case the direct sum of {\em all} indecomposables is a generator-cogenerator $G$ which satisfies $\dim_\cX(\Lambda) =0$ for $\cX = \add \; G = \mods \Lambda$. To show that $\repdim \Lambda \le 2$ in this case, 
 Auslander used the following argument : 
For any right module $T$ over $\End_{\Lambda}(G)$
consider a projective presentation
\[
\Hom_\Lambda(G,G'') \to \Hom_\Lambda(G,G') \to T \to 0
\]
where $G', G'' \in \add \; G$ (note that the projective $\End_{\Lambda}(G)$-modules are of the form $\Hom_\Lambda(G,H)$ for some $H \in \add \; G$).
The map from $\Hom_\Lambda(G,G'')$ to $\Hom_\Lambda(G,G')$ is induced by a map $f:G'' \to G'$, and the exact sequence
\[
0 \to \Ker f \to G'' \to G'
\] yields the desired projective resolution 
\[
0 \to \Hom_\Lambda(G, \Ker f) \to \Hom_\Lambda(G,G'') \to \Hom_\Lambda(G,G') \to T \to 0,
\]
showing that the global dimension of $\End_\Lambda(G)$ is at most two. 
The work \cite{EHIS} extended Auslander's argument by adding tails of $\cX$-resolutions to $\Ker f$ of the form 
\[
0 \to X_n \to \ldots X_0 \to \Ker f \to 0
\]
to get the result for $n \ge 0$ which Auslander showed for $n=0$.

\begin{example}
    Consider the local algebra $\Lambda = K[x]/(x^2)$. This is a self-injective algebra with two classes of indecomposable modules. Thus the free module $\Lambda$ is a generator-cogenerator with $\gldim(\End_\Lambda(\Lambda)) = \infty$. On the other hand, Auslander's argument summarized above shows that $\repdim\Lambda = 2$. This demonstrates that adding extra direct summands to a given generator-cogenerator $G$ can have the effect of lowering the global dimension of the endomorphism ring.
\end{example}

Apart from the representation-finite case, the representation-dimension is in general hard to determine. In fact, the first examples of representation-infinite algebras whose representation-dimension was determined all satisfy $\repdim \Lambda = 3$, for instance when $\Lambda$ is a hereditary or tilted algebra. We refer to \cite{Xi2000} for more examples of algebras whose representation dimension could be determined.
The first examples of arbitrarily large representation dimension were discovered in \cite{rouquier}.
\medskip 

We aim to illustrate how the notion of representation-dimension compares to the $\add(G)$-global dimension of $\field_\cP$ discussed above.
Lemma 2.1 in \cite{EHIS} can be summarised as follows:
\begin{proposition}\label{prop:generator_cogenerator}
    Assume $G$ is a generator-cogenerator of $\rep \cP$ and $\cX = \add \; G$. 
Then, for all $n \ge 1$,
    $\dim_\cX(\field_\cP) = n \mbox{ if and only if } \gldim(\End_{\cP}(G)) = n +2.
    $
\end{proposition}

\begin{remark}\label{rem:generator}
    Only one direction of the proof of Proposition~\ref{prop:generator_cogenerator} given in \cite{EHIS} uses the fact that $G$ is a cogenerator. Thus for $G$ a generator of $\rep \cP$ and $\cX = \add \ G$, we have that $\dim_\cX(\field_\cP) \leq \gldim (\End_{\cP}(G)) \leq \dim_\cX(\field_\cP) + 2$. In particular, $\dim_\cX(\field_\cP)$ is infinite if and only if $\gldim(\End_{\cP}(G))$ is infinite. 
\end{remark}

Returning to the case where $G$ is a generator-cogenerator, the only case preventing us from saying that the endomorphism ring of $G$ has global dimension equal $\dim_\cX(\field_\cP) + 2$ is when the ring $\End_{\cP} G$ is semi-simple (which results in global dimension zero).
We assume from now that this is not the case.

When computing homological invariants for a persistence module, one would like to establish a lower bound on the length of a minimal resolution, thus bounding the complexity of the computation. The representation dimension provides such a bound.

This motivates the study of representation dimension for posets. However, it is very difficult in general to determine the representation dimension. A general result by Iyama states that the representation dimension of an artin algebra is always finite \cite{iyama}. This result has been shown earlier in \cite{Xi2002} for the case of posets, which is what we consider here.
However, finite is not good enough a bound for concrete applications, 
so we illustrate the case of multiparameter persistence in more detail:
An $n$-dimensional finite grid is a finite poset of the form 
 $\cP = \cT_1\times \cdots \times \cT_n$ where each $\cT_i$ is a totally ordered set, equipped the usual product order.
Then the incidence algebra $\field_\cP$ can be realized as a tensor product (over $\field$):
$
\field_\cP = \field_{\cT_1 }\otimes \cdots \otimes \field_{\cT_n}
$,
 where each $\field_{\cT_i}$ is the path algebra of a linear oriented quiver of type $A$.
We apply then the following result from \cite[Theorem 3.5]{Xi2000}:
\begin{theorem}\label{thm:repdim-tensor}
    Suppose $\field$ is a perfect field.  Then
$
\repdim\field_\cP \le \sum_{i = 1}^n \repdim \field_{\cT_i }
$.
\end{theorem}
Since each each $\field_{\cT_i}$ is representation-finite, this yields $\repdim\field_\cP \le 2n$, in particular for a 2-dimensional grid we have $\repdim\field_\cP \le 4.$
On the other hand, except for some very small cases, the algebra $\field_\cP$ is representation-infinite, so we know $3 \le \repdim\field_\cP \le 4$ for the general $2d$-grid. 

\begin{remark}\label{rem:grid}
    It would be very interesting to determine the representation dimension of the $n$-dimensional grid.
\end{remark}

The proof of Theorem \ref{thm:repdim-tensor} uses as generator-cogenerator a module $G$ which consists of sums of tensors $\mathbb{I}_1 \otimes \cdots \otimes \mathbb{I}_n$, thus for the case of an $n$-dimensional finite grid $\cP$, the set $\cX = \add G$ consists of all direct sums of segment modules. The argument from \cite{Xi2000} then shows that $\dim_\cX(\field_\cP) = 2n-2$. We discuss how this extends to infinite posets in Example~\ref{ex:infinite_exact}(3). We also note that, for $\cP$ finite, the resolution in \cite[Section~H]{CGRST} has length at most $2n$, and thus has 0 homology in degrees $2n$ and $2n-1$.

The question is whether the estimate $\repdim\Lambda \leq 2n$ can be lowered using a different set of modules. In light of the following result \cite[Corollary 9.4]{ringel2012}, it seems unlikely that the bound of $2n$ is tight:
\begin{theorem}\label{thm:repdim-ringel}
Let $\Lambda_1, \ldots, \Lambda_n$ be path algebras of representation-infinite bipartite quivers. Then the algebra $\Lambda_1 \otimes \cdots \otimes \Lambda_n$ has representation dimension precisely $n + 2.$
\end{theorem}


\section{Characterizations of irreducible morphisms}\label{sec:irreducible}

In this section, we describe the $\cX$-irreducible morphisms (see Definition~\ref{def:irreducible}) for some choices of $\cX$. The statement of each characterization is accompanied by an illustrated example over a finite rectangular subposet of $\mathbb{N}^2$. More precisely, each illustration describes all possible $\cX$-irreducible morphisms $\mathbb{I}_\cS \rightarrow \mathbb{I}_\cT$ for some fixed choice of $\cS$. The region $\cS \cap \cT$ is shaded gray, the region $\cS \setminus \cT$ shaded orange, and the region $\cT \setminus \cS$ shaded blue. The first example can be found after the statement of Proposition~\ref{prop:upset}.

We assume throughout this section that $\cP$ is finite and that $\cX \subseteq \ind(\rep \cP)$ is finite and satisfies $\proj(\rep \cP) \subseteq \add(\cX)$. Thus $G = \bigoplus_{X \in \cX} X$ is a generator (but not necessarily a cogenerator) of $\rep \cP$ as in Section~\ref{sec:rep_dim}.

As mentioned in the introduction, knowledge of the $\cX$-irreducible morphisms between indecomposables is enough to compute the Gabriel quiver of the $\field$-algebra $\End_{\cP}(G)$. See e.g. \cite[Section~4.3]{BBH}. This can be a useful tool for bounding the global dimension of $\End_{\cP}(G)$, and thus by Remark~\ref{rem:generator} also the $\add(\cX)$-global dimension of $\field_\cP$. Example~\ref{ex:spreads_infinite} at the end of this section gives a detailed example of this process.

We begin with the definition.

\begin{definition}\label{def:irreducible}
        A morphism $f:X\to X'$, with $X,X'\in \cX$, is called $\cX$-\textit{irreducible}, or\textit{ irreducible relative to} $\cX$,  if $f$ is not a section or a retraction and all factorisations of $f$ of the shape $X\xto{g} U \xto{h} X'$, with $U\in \text{add} \cX$, must have $g$ a section or $h$ a retraction. 
\end{definition}

The results and proofs of this section make multiple references to the characterization of hom-spaces in Proposition~\ref{prop:hom_spread} and the notion of covers from Definition~\ref{def:covers}. The following lemma will also be critical. See also \cite[Proposition~5.10]{BBH}.

\begin{lemma}\label{lem:irreducible} Let $X \neq Y \in \cX$.
    \begin{enumerate}
        \item Suppose that $\End_\cP(X) \cong \field \cong \End_\cP(Y)$ and  $\dim_\field\Hom_\cP(X,Y) = 1$. Then $\Hom_\cP(X,Y)$ contains a (necessarily unique up to scalar multiplication) irreducible morphism if and only if there do not exist $Z \in \cX \setminus \{X,Y\}$ and morphisms $X \xrightarrow{g} Z \xrightarrow{h} Y$ with $h \circ g \neq 0$.
        \item Suppose there exists an $\cX$-irreducible morphism $f: X \rightarrow Y$. If $\im f \in \add(\cX)$, then $f$ must be either injective or surjective.
    \end{enumerate}
\end{lemma}

\begin{proof}
    (1) Suppose first that there exists $Z \in \cX\setminus \{X,Y\}$ and morphisms $X \xrightarrow{g} Z \xrightarrow{h} Y$ with $h \circ g \neq 0$. Then $h \circ g$ is a basis element of $\Hom_\field(X,Y) \cong \field$. Since $X, Y$, and $Z$ are all indecomposable, it follows that $\Hom_\cP(X,Y)$ contains no $\cX$-irreducible morphism.

    Now suppose no such $Z$ exists, and let $f: X\rightarrow Y$ be nonzero. Suppose $f$ admits a factorisation $X \xrightarrow{g} U \xrightarrow{h} Y$ with $U \in \add(\cX)$. Then for every indecomposable direct summand $Z$ of $U$, either the induced map $X \rightarrow Z \rightarrow Y$ is zero or $Z \in \{X,Y\}$. Since $X$ and $Y$ have trivial endomorphism rings, it follows that $g$ is a section or $h$ is a retraction.
    
    (2) Suppose $f: X \rightarrow Y$ is $\cX$-irreducible. Since $\im f \in \add(\cX)$, either the quotient map $X \twoheadrightarrow \im f$ is a section (in which case $f$ is injective) or the inclusion map $\im f\hookrightarrow Y$ is a retraction (in which case $f$ is surjective).
\end{proof}

\begin{proposition}\label{prop:upset}
Suppose $\cP$ has a unique maximal element, and let $\cX$ be the set of upset modules. Then, for $\mathbb{I}_\cS, \mathbb{I}_\cT \in \cX$, we have $\Hom_\cP(\mathbb{I}_\cS,\mathbb{I}_\cT) \cong \field$ if $\cS\subseteq \cT$ and $\Hom_\cP(\mathbb{I}_\cS,\mathbb{I}_\cT) = 0$ otherwise. In particular, there is an $\cX$-relative irreducible morphism $\mathbb{I}_\cT\to \mathbb{I}_\cS$ if and only if $\cS=\cT\cup \{x\}$ for some $x\cover \cS$.
\end{proposition}

\begin{center}
    \begin{tikzpicture}[scale = 0.75]
        \filldraw (0,0) circle (0.1);
        \filldraw (1,0) circle (0.1);
        \filldraw (2,0) circle (0.1);
        \filldraw (0,1) circle (0.1);
        \filldraw (1,1) circle (0.1);
        \filldraw (2,1) circle (0.1);
        \filldraw (0,2) circle (0.1);
        \filldraw (1,2) circle (0.1);
        \filldraw (2,2) circle (0.1);
        \draw[fill,color=black,opacity=.2] (-0.5,1.5)--(-0.5,2.5)--(2.5,2.5)--(2.5,-0.5)--(1.5,-0.5)--(1.5,0.5)--(0.5,0.5)--(0.5,1.5)--cycle;
        \draw[fill,color=blue,opacity=.2] (-0.5,0.5)--(-0.5,1.5)--(0.5,1.5)--(0.5,0.5)--cycle;
        \node [anchor = north east] at (0,1) {$x$};
    \begin{scope}[shift = {(5,0)}]
        \filldraw (0,0) circle (0.1);
        \filldraw (1,0) circle (0.1);
        \filldraw (2,0) circle (0.1);
        \filldraw (0,1) circle (0.1);
        \filldraw (1,1) circle (0.1);
        \filldraw (2,1) circle (0.1);
        \filldraw (0,2) circle (0.1);
        \filldraw (1,2) circle (0.1);
        \filldraw (2,2) circle (0.1);
        \draw[fill,color=black,opacity=.2] (-0.5,1.5)--(-0.5,2.5)--(2.5,2.5)--(2.5,-0.5)--(1.5,-0.5)--(1.5,0.5)--(0.5,0.5)--(0.5,1.5)--cycle;
        \draw[fill,color=blue,opacity=.2] (0.5,-0.5)--(0.5,0.5)--(1.5,0.5)--(1.5,-0.5)--cycle;
        \node [anchor = north east] at (1,0) {$x$};
    \end{scope}
    \end{tikzpicture}
\end{center}

While we provide a proof here for clarity and similarity to the other cases, this result was previously shown as part of Miller's work on upsets and downsets, see \cite[Proposition~3.10]{miller}. See also \cite[Example~B]{CGRST}.

\begin{proof} First suppose $\cS\subseteq \cT$. Since $\cP$ has a unique maximum, $\cS\cap \cT= \cS$ is connected.
It follows from Proposition~\ref{prop:hom_spread} that $\Hom_\cP(\mathbb{I}_\cS,\mathbb{I}_\cT) \cong \field$.

Now suppose 
that $\cS \not\subseteq \cT$. Then there exists $x\in \cS\setminus \cT$. 
Now the unique maximum of $\cP$ is included in every upset, so the intersection $\cS\cap \cT$ is connected and nonempty. As $x \leq \cS\cap \cT$, Proposition~\ref{prop:hom_spread} implies that $\Hom_\cP(\mathbb{I}_\cS,\mathbb{I}_\cT) = 0$.

As we just proved that the existence of nonzero morphisms between upset modules induces the same order as inclusion for upsets, and that all nontrivial hom-spaces have dimension 1. It follows from Lemma~\ref{lem:irreducible}(1) that the $\cX$-irreducible morphisms are the cover relations, namely adding a single element at a time.  \end{proof}

For $\cX= \{ \mathbb{I}_{\langle a,b\rangle} | \; a  \leq  b\in \cP\}$ the set of segment modules, it is well known that there are non-zero morphisms $\mathbb{I}_{\langle a,b\rangle} \to \mathbb{I}_{\langle c,d \rangle}$ if and only if $c\leq a\leq  d \leq b$, and that all nonzero hom-spaces must have dimension 1. See for example \cite[Corollary~5.8]{BBH}. The $\cX$-irreducible morphisms can be described using this fact. \textit{Cf.} \cite[Example~H]{CGRST} in the case that $\cP$ is an $n$-dimensional grid.

\begin{proposition}\label{prop:morphismforsegments}
    Let $\cX= \left\{ \mathbb{I}_{\langle a,b\rangle} | \; a \leq b\in \cP\right\}$ be the set of segment modules. There is an $\cX$-irreducible morphism $f: \mathbb{I}_{\langle a,b\rangle }\rightarrow \mathbb{I}_{\langle c,d \rangle}$ for $ \langle a,b\rangle, \langle c,d\rangle \in \cX$ if and only if one of the following is verified:
    \begin{enumerate}
        \item $f$ is injective and $c\cover a\leq d=b$.
        \item $f$ is surjective and $c=a\leq d\cover b$.
    \end{enumerate}
\end{proposition}

\begin{center}
    \begin{tikzpicture}[scale = 0.75]
        \filldraw (0,0) circle (0.1);
        \filldraw (1,0) circle (0.1);
        \filldraw (2,0) circle (0.1);
        \filldraw (0,1) circle (0.1);
        \filldraw (1,1) circle (0.1);
        \filldraw (2,1) circle (0.1);
        \filldraw (0,2) circle (0.1);
        \filldraw (1,2) circle (0.1);
        \filldraw (2,2) circle (0.1);
        \draw[fill,color=black,opacity=.2] (0.5,-0.5)--(0.5,2.5)--(2.5,2.5)--(2.5,-0.5)--cycle;
        \draw[fill,color=blue,opacity=.2] (-0.5,-0.5)--(-0.5,2.5)--(0.5,2.5)--(0.5,-0.5)--cycle;
        \node [anchor = north] at (1,0) {$a$};
        \node [anchor = north] at (0,0) {$c$};
        \node [anchor = south] at (2,2) {$b = d$};
    \begin{scope}[shift = {(5,0)}]
        \filldraw (0,0) circle (0.1);
        \filldraw (1,0) circle (0.1);
        \filldraw (2,0) circle (0.1);
        \filldraw (0,1) circle (0.1);
        \filldraw (1,1) circle (0.1);
        \filldraw (2,1) circle (0.1);
        \filldraw (0,2) circle (0.1);
        \filldraw (1,2) circle (0.1);
        \filldraw (2,2) circle (0.1);
        \draw[fill,color=black,opacity=.2] (0.5,-0.5)--(0.5,1.5)--(2.5,1.5)--(2.5,-0.5)--cycle;
        \draw[fill,color=orange,opacity=.2] (0.5,1.5)--(0.5,2.5)--(2.5,2.5)--(2.5,1.5)--cycle;
        \node [anchor = south] at (2,2) {$b$};
        \node [anchor = west] at (2,1) {$d$};
        \node [anchor = north] at (1,0) {$a = c$};
    \end{scope}
     \begin{scope}[shift = {(10,0)}]
        \filldraw (0,0) circle (0.1);
        \filldraw (1,0) circle (0.1);
        \filldraw (2,0) circle (0.1);
        \filldraw (0,1) circle (0.1);
        \filldraw (1,1) circle (0.1);
        \filldraw (2,1) circle (0.1);
        \filldraw (0,2) circle (0.1);
        \filldraw (1,2) circle (0.1);
        \filldraw (2,2) circle (0.1);
        \draw[fill,color=black,opacity=.2] (0.5,-0.5)--(0.5,2.5)--(1.5,2.5)--(1.5,-0.5)--cycle;
        \draw[fill,color=orange,opacity=.2] (1.5,-0.5)--(1.5,2.5)--(2.5,2.5)--(2.5,-0.5)--cycle;
        \node [anchor = south] at (2,2) {$b$};
        \node [anchor = south] at (1,2) {$d$};
        \node [anchor = north] at (1,0) {$a = c$};
    \end{scope}
    \end{tikzpicture}
\end{center}

\begin{proof} If $a < c$ or $ b < d$ then $\Hom_{\cP}\left(\mathbb{I}_{\langle a,b\rangle},\mathbb{I}_{\langle c,d\rangle}\right) = 0$ and we are done. Thus suppose that $c \leq a$ and $d \leq b$  and let $f: \mathbb{I}_{\langle a,b\rangle }\rightarrow \mathbb{I}_{\langle c,d \rangle}$ be nonzero. If $a=c$ and $b=d$, we have an endomorphism which is forced to be an isomorphism as segment modules are bricks.

Suppose $a\neq c$. Then $c < a$, and there exists $c'$ such that $c\leq c' \cover a$. Then $f$ has a factorisation of the shape $\mathbb{I}_{\langle a,b\rangle}\to \mathbb{I}_{\langle c',b\rangle}\to \mathbb{I}_{\langle c,d\rangle}$, which is trivial if and only if $f$ is irreducible.  Suppose the factorisation is trivial. The first morphism is not an isomorphism as $c'\notin \langle a,b\rangle$, so the second one must be. By Lemma~\ref{lem:irreducible}(1), we need only consider factorisations of the shape $\mathbb{I}_{\langle a,b\rangle} \to \mathbb{I}_{\langle x,y\rangle} \to \mathbb{I}_{\langle c',b\rangle}$. But then $b\leq y \leq b$, so $b=y$, and the second morphism is a non-isomophism if and only if $c'< x < a$. We conclude that, in this case, $f$ is irreducible if and only if $c \cover a \leq d = b$. Moreover, we have $\langle a,b\rangle \subseteq \langle c,d\rangle$, so $f$ is injective.

The case where $b \neq d$ is similar, and in that case the fact that $\langle a,b\rangle \supseteq \langle c,d\rangle$ implies that $f$ is surjective. \end{proof}

We next consider hook modules, \textit{cf.} \cite[Example~F]{CGRST} in the case that $\cP$ is a join-semilattice.

\begin{proposition}
    Let $\cX=\{\mathbb{I}_{\langle a,b\langle} | a\leq b\in \cP \}$ be the set of hook modules. Then $\Hom_\cP(\mathbb{I}_{\langle a,b\langle},\mathbb{I}_{\langle c,d\langle}) \cong \field$ if $c \leq a \not\geq d \leq b$ and $\Hom_\cP(\mathbb{I}_{\langle a,b\langle},\mathbb{I}_{\langle c,d\langle}) = 0$ otherwise. Moreover, there is an $\cX$-irreducible morphism $f: \mathbb{I}_{\langle a,b\langle} \rightarrow \mathbb{I}_{\langle c,d\langle} $ if and only if one of the following is verified:

    \begin{enumerate}
        \item $f$ is injective and $c\cover a \not\geq d= b$.
        \item $f$ is surjective and $c=a\not\geq d\cover b$.
    \end{enumerate}
\end{proposition}

\begin{center}
    \begin{tikzpicture}[scale = 0.75]
        \filldraw (0,0) circle (0.1);
        \filldraw (1,0) circle (0.1);
        \filldraw (2,0) circle (0.1);
        \filldraw (-1,0) circle (0.1);
        \filldraw (0,1) circle (0.1);
        \filldraw (1,1) circle (0.1);
        \filldraw (2,1) circle (0.1);
        \filldraw (-1,1) circle (0.1);
        \filldraw (0,2) circle (0.1);
        \filldraw (1,2) circle (0.1);
        \filldraw (2,2) circle (0.1);
        \filldraw (-1,2) circle (0.1);
        \draw[fill,color=black,opacity=.2] (-0.5,-0.5)--(-0.5,2.5)--(1.5,2.5)--(1.5,0.5)--(2.5,0.5)--(2.5,-0.5)--cycle;
        \draw[fill,color=blue,opacity=.2] (-1.5,-0.5)--(-1.5,2.5)--(-0.5,2.5)--(-0.5,-0.5)--cycle;
        \node [anchor = north] at (0,0) {$a$};
        \node [anchor = north] at (-1,0) {$c$};
        \node [anchor = west] at (2,1) {$b = d$};
    \begin{scope}[shift = {(5.5,0)}]
        \filldraw (0,0) circle (0.1);
        \filldraw (1,0) circle (0.1);
        \filldraw (2,0) circle (0.1);
        \filldraw (-1,0) circle (0.1);
        \filldraw (0,1) circle (0.1);
        \filldraw (1,1) circle (0.1);
        \filldraw (2,1) circle (0.1);
        \filldraw (-1,1) circle (0.1);
        \filldraw (0,2) circle (0.1);
        \filldraw (1,2) circle (0.1);
        \filldraw (2,2) circle (0.1);
        \filldraw (-1,2) circle (0.1);
        \draw[fill,color=black,opacity=.2] (-0.5,-0.5)--(-0.5,2.5)--(0.5,2.5)--(0.5,0.5)--(2.5,0.5)--(2.5,-0.5)--cycle;
        \draw[fill,color=orange,opacity=.2] (0.5,0.5)--(0.5,2.5)--(1.5,2.5)--(1.5,0.5)--cycle;
        \node [anchor = south] at (2,1) {$b$};
        \node [anchor = north] at (0,0) {$a = c$};
        \node [anchor = south] at (1,1) {$d$};
    \end{scope}
     \begin{scope}[shift = {(11,0)}]
        \filldraw (0,0) circle (0.1);
        \filldraw (1,0) circle (0.1);
        \filldraw (2,0) circle (0.1);
        \filldraw (-1,0) circle (0.1);
        \filldraw (0,1) circle (0.1);
        \filldraw (1,1) circle (0.1);
        \filldraw (2,1) circle (0.1);
        \filldraw (-1,1) circle (0.1);
        \filldraw (0,2) circle (0.1);
        \filldraw (1,2) circle (0.1);
        \filldraw (2,2) circle (0.1);
        \filldraw (-1,2) circle (0.1);
        \draw[fill,color=black,opacity=.2] (-0.5,-0.5)--(-0.5,2.5)--(1.5,2.5)--(1.5,0.5)--(1.5,-0.5)--cycle;
        \draw[fill,color=orange,opacity=.2] (1.5,-0.5)--(1.5,0.5)--(2.5,0.5)--(2.5,-0.5)--cycle;
        \node [anchor = west] at (2,1) {$b$};
        \node [anchor = north] at (0,0) {$a = c$};
        \node [anchor = west] at (2,0) {$d$};
    \end{scope}
    \end{tikzpicture}
\end{center}

\begin{proof}To have a non-zero morphism, one needs $a\in \langle c,d\langle$, so $c\leq a \not\geq d$. If $d\not \leq b$, then $d\in \langle a,b\langle$ and any morphism is forced to be zero by Proposition~\ref{prop:hom_spread}. Otherwise, the hom-space is one-dimensional, again by Proposition~\ref{prop:hom_spread}.
The argument for relative irreducibles is analogous to the case of segments of Proposition~\ref{prop:morphismforsegments}.
\end{proof}

Before considering single-source spread modules, we need the following.

\begin{lemma}\label{lem:single_source}
     In a finite poset $\cP$, for every single-source spread $\langle a,B\rangle$, there exist some $B'\subseteq \langle a,\infty\langle$ such that $\langle a,B\rangle= \langle a,B' \langle$.
\end{lemma}

\begin{proof} First, $\langle a,\infty\langle \ \setminus \ \langle a,B\rangle$ is a connected upset, so it must be equal to some $\langle B',\infty\langle$, and as $\cP$ is finite, so is $B'$. Now
$$\langle a,B\rangle \cup \langle B', \infty \langle = \langle a,\infty \langle = \langle a,B'\langle \ \cup \ \langle B',\infty \langle$$ where both unions are disjoint, so $\langle a,B\rangle = \langle a,B'\langle$ as required.
\end{proof}

We now consider single-source spread modules, \textit{cf.} \cite[Example~E]{CGRST}.

\begin{proposition}\label{prop:irreducibles_single_source}
    Let $\cX$ be the set of single-source spread modules. For all $\mathbb{I}_{\langle a,B\rangle} \in \cX$, let $B'$ be such that $\langle a,B\rangle=\langle a,B'\langle$ as in Lemma~\ref{lem:single_source}. Then there is an $\cX$-irreducible morphism $f: \mathbb{I}_{\langle a,B\rangle} \rightarrow \mathbb{I}_{\langle c,D\rangle}$ if and only if one of the following is verified:
    \begin{enumerate}
        \item $f$ is surjective and $\langle c,D\rangle=\langle a,B\rangle\setminus \{b\}$ for some $b\in B$.
        \item $f$ is injective, $c\cover a$ and $B'=D'$.
    \end{enumerate}
\end{proposition}

\begin{center}
    \begin{tikzpicture}[scale = 0.75]
        \filldraw (0,0) circle (0.1);
        \filldraw (1,0) circle (0.1);
        \filldraw (2,0) circle (0.1);
        \filldraw (-1,0) circle (0.1);
        \filldraw (-2,0) circle (0.1);
        \filldraw (0,1) circle (0.1);
        \filldraw (1,1) circle (0.1);
        \filldraw (2,1) circle (0.1);
        \filldraw (-1,1) circle (0.1);
        \filldraw (-2,1) circle (0.1);
        \filldraw (0,2) circle (0.1);
        \filldraw (1,2) circle (0.1);
        \filldraw (2,2) circle (0.1);
        \filldraw (-1,2) circle (0.1);
        \filldraw (-2,2) circle (0.1);
        \draw[fill,color=black,opacity=.2] (-1.5,-0.5)--(-1.5,1.5)--(-0.5,1.5)--(0.5,1.5)--(0.5,0.5)--(2.5,0.5)--(2.5,-0.5)--cycle;
        \draw[fill,color=orange,opacity=.2] (-1.5,1.5)--(-1.5,2.5)--(-0.5,2.5)--(-0.5,1.5)--cycle;
        \node [anchor = north] at (-1,0) {$a = c$};
        \node [anchor = south west] at (-1,2) {$b$};
    \begin{scope}[shift = {(8,0)}]
        \filldraw (0,0) circle (0.1);
        \filldraw (1,0) circle (0.1);
        \filldraw (2,0) circle (0.1);
        \filldraw (-1,0) circle (0.1);
        \filldraw (-2,0) circle (0.1);
        \filldraw (0,1) circle (0.1);
        \filldraw (1,1) circle (0.1);
        \filldraw (2,1) circle (0.1);
        \filldraw (-1,1) circle (0.1);
        \filldraw (-2,1) circle (0.1);
        \filldraw (0,2) circle (0.1);
        \filldraw (1,2) circle (0.1);
        \filldraw (2,2) circle (0.1);
        \filldraw (-1,2) circle (0.1);
        \filldraw (-2,2) circle (0.1);
        \draw[fill,color=black,opacity=.2] (-1.5,-0.5)--(-1.5,2.5)--(-0.5,2.5)--(-0.5,1.5)--(0.5,1.5)--(0.5,0.5)--(2.5,0.5)--(2.5,-0.5)--cycle;
        \draw[fill,color=orange,opacity=.2] (0.5,0.5)--(0.5,1.5)--(1.5,1.5)--(1.5,0.5)--cycle;
        \node [anchor = north] at (-1,0) {$a = c$};
        \node [anchor = south west] at (1,1) {$b$};
    \end{scope}
     \begin{scope}[shift = {(0,-4)}]
        \filldraw (0,0) circle (0.1);
        \filldraw (1,0) circle (0.1);
        \filldraw (2,0) circle (0.1);
        \filldraw (-1,0) circle (0.1);
        \filldraw (-2,0) circle (0.1);
        \filldraw (0,1) circle (0.1);
        \filldraw (1,1) circle (0.1);
        \filldraw (2,1) circle (0.1);
        \filldraw (-1,1) circle (0.1);
        \filldraw (-2,1) circle (0.1);
        \filldraw (0,2) circle (0.1);
        \filldraw (1,2) circle (0.1);
        \filldraw (2,2) circle (0.1);
        \filldraw (-1,2) circle (0.1);
        \filldraw (-2,2) circle (0.1);
        \draw[fill,color=black,opacity=.2] (-1.5,-0.5)--(-1.5,2.5)--(-0.5,2.5)--(-0.5,1.5)--(0.5,1.5)--(0.5,0.5)--(1.5,0.5)--(1.5,-0.5)--cycle;
        \draw[fill,color=orange,opacity=.2] (1.5,-0.5)--(1.5,0.5)--(2.5,0.5)--(2.5,-0.5)--cycle;
        \node [anchor = north] at (-1,0) {$a = c$};
        \node [anchor = north west] at (2,0) {$b$};
    \end{scope}
    \begin{scope}[shift = {(8,-4)}]
        \filldraw (0,0) circle (0.1);
        \filldraw (1,0) circle (0.1);
        \filldraw (2,0) circle (0.1);
        \filldraw (-1,0) circle (0.1);
        \filldraw (-2,0) circle (0.1);
        \filldraw (0,1) circle (0.1);
        \filldraw (1,1) circle (0.1);
        \filldraw (2,1) circle (0.1);
        \filldraw (-1,1) circle (0.1);
        \filldraw (-2,1) circle (0.1);
        \filldraw (0,2) circle (0.1);
        \filldraw (1,2) circle (0.1);
        \filldraw (2,2) circle (0.1);
        \filldraw (-1,2) circle (0.1);
        \filldraw (-2,2) circle (0.1);
        \draw[fill,color=black,opacity=.2] (-1.5,-0.5)--(-1.5,2.5)--(-0.5,2.5)--(-0.5,1.5)--(1.5,1.5)--(1.5,0.5)--(2.5,0.5)--(2.5,-0.5)--cycle;
        \draw[fill,color=blue,opacity=.2] (-2.5,-0.5)--(-2.5,2.5)--(-1.5,2.5)--(-1.5,-0.5)--cycle;
        \node [anchor = north] at (-1,0) {$c$};
        \node [anchor = north] at (-2,0) {$a$};
    \end{scope}
    \end{tikzpicture}
\end{center}

\begin{proof} First note that the hom-space between single-source spread modules has dimension at most one by Proposition~\ref{prop:hom_spread}, and so Lemma~\ref{lem:irreducible}(1) applies. Now let $f: \mathbb{I}_{\langle a,B\rangle} \to \mathbb{I}_{\langle c,D\rangle}$ be a non-zero
morphism. Then $a\in \langle c,D\rangle$, so $c\leq a$. 

First suppose $c=a$. Then by Proposition~\ref{prop:hom_spread}, we must have $\langle a,B\rangle \supseteq \langle a,D\rangle$ and $f$ is surjective. Now suppose $f$ is irreducible. Then we cannot have $B=D$ as it would lead to an isomorphism. Thus there exists an element $x\in \langle a,B\rangle \setminus \langle a,D\rangle$, and some $b\in B$ such that $x\leq b$. (Note that $b \notin D$ because $d \notin \langle a,D\rangle$.) There is a factorisation of the shape $\mathbb{I}_{\langle a,B\rangle} \to \mathbb{I}_{\langle a,B\rangle\setminus \{b\}} \to \mathbb{I}_{\langle a,D\rangle }$, which must be trivial, so $\langle c,D\rangle=\langle a,D\rangle = \langle a,B\rangle \setminus \{b\}$. Reciprocally, suppose condition (1) is verified, 
and suppose for a contradiction that there exists a proper factorisation $\mathbb{I}_{\langle a,B\rangle} \to \mathbb{I}_{\langle x,Y\rangle} \to \mathbb{I}_{\langle a,B\rangle \setminus \{b\}}$. Then $a\leq x \leq a$ so $x=a$. As we have morphisms, we have $\langle a,B\rangle \setminus \{b\}\subseteq \langle a,Y\rangle \subseteq \langle a,B\rangle$ but as the left and right term differ by a single element, one of these two inclusions must be an equality. In each case, a morphism of the supposedly proper factorisation is an isomorphism, a contradiction.

Now suppose $c < a$. By Proposition~\ref{prop:hom_spread}, $\im f$ is a single-source spread module, and so it suffices to consider only the case that $f$ is injective by Lemma~\ref{lem:irreducible}(2); that is, we can assume that $\langle c,D\rangle \subseteq \langle a,B\rangle$. Let $a'$ be a maximal element of $\{x\in \cP | \; c\leq x < a\}$.  Then $a'\in \langle c,D\rangle$. Then $f$ has a factorisation of the shape $\mathbb{I}_{\langle a, B'\langle} \to \mathbb{I}_{\langle a',B'\langle} \xto{g} \mathbb{I}_{\langle c,D'\langle} $. The first morphism is simply induced by inclusion and $g$ is induced by $g_{a'}=f_a$. The left morphism cannot be invertible as $a' \not \in \langle a,B'\langle$. If $f$ is $\cX$-irreducible, $g$ must then be invertible, so $\langle c,D'\langle \; = \langle a', B'\langle$. Then $c\cover a$ and $B'=D'$ as required in condition (2).  Reciprocally, suppose condition (2) is verified,
and suppose for a contradiction that there is a proper factorisation of $f$ of the shape $\mathbb{I}_{\langle a,B' \langle} \xto{h} \mathbb{I}_{\langle x,Y'\langle} \to \mathbb{I}_{\langle c,B'\langle}$. So $c\leq x \leq a$. But then as $c \cover a$, we have $c=x$ or $x=a$.  In the first case, we have $\langle x,Y'\langle \; =\langle c,Y'\langle$. As $h$ is non-zero, the linear maps $\mathbb{I}_{\langle x,Y'\rangle}(a,z)$ must be non-zero for all $z\in \langle a,B'\langle$. Indeed, suppose there exist a $z$ such that $\mathbb{I}_{\langle x,Y'\rangle}(a,z)=0$. Then the commutativity condition $h_z\circ \mathbb{I}_{\langle a,B\rangle} (a,z)= \mathbb{I}_{\langle x,Y'\langle}(a,z) \circ h_a$ implies that $h_z\circ 1=0$, so $h_z=0$, which implies $h_a=0$ and $h=0$, a contradiction. This implies $\langle Y',\infty\langle \supseteq \langle B',\infty\langle$. The right morphism similarly imposes $\langle Y',\infty \langle \; \subseteq \langle B', \infty\langle$, so $Y'=B'$. The second case where $x=a$ is similar.
\end{proof}

Finally let's look at irreducible relative to the set of all spread modules.

\begin{proposition}\label{prop:irreduciblesbetweenspreads}
    Let $\cX=\{\mathbb{I}_{\langle A,B\rangle}| \; A,B\subseteq \cP\}$ be the set of spread modules. There exist an $\cX$-irreducible morphism $f: \mathbb{I}_{\langle A,B\rangle}\rightarrow \mathbb{I}_{\langle C,D\rangle}$ if and only if one of the following holds:
    \begin{enumerate}
    \item $f$ is surjective and $\langle A,B\rangle = \langle C,D\rangle \ \cup\; \rangle D,x\rangle$ for some $\langle C,D\rangle \cover x$.

    \item $f$ is injective and $\langle A,B\rangle$ is a connected component of $\langle C,D\rangle\setminus \{c\}$ for some $c\in C$ such that $\langle C,D\rangle= \langle A,B\rangle \cup \langle c,A\langle$. 
    \end{enumerate}
\end{proposition}

\begin{center}
    \begin{tikzpicture}[scale = 0.75]
        \filldraw (0,0) circle (0.1);
        \filldraw (1,0) circle (0.1);
        \filldraw (2,0) circle (0.1);
        \filldraw (-1,0) circle (0.1);
        \filldraw (0,1) circle (0.1);
        \filldraw (1,1) circle (0.1);
        \filldraw (2,1) circle (0.1);
        \filldraw (-1,1) circle (0.1);
        \filldraw (0,2) circle (0.1);
        \filldraw (1,2) circle (0.1);
        \filldraw (2,2) circle (0.1);
        \filldraw (-1,2) circle (0.1);
        \filldraw (0,-1) circle (0.1);
        \filldraw (1,-1) circle (0.1);
        \filldraw (2,-1) circle (0.1);
        \filldraw (-1,-1) circle (0.1);
        \draw[fill,color=black,opacity=.2] (0.5,-0.5)--(0.5,0.5)--(-0.5,0.5)--(-0.5,1.5)--(1.5,1.5)--(1.5,-0.5)--cycle;
        \draw[fill,color=orange,opacity=.2] (-1.5,1.5)--(-1.5,2.5)--(0.5,2.5)--(0.5,1.5)--cycle;
        \node [anchor = south west] at (0,2) {$x$};
    \begin{scope}[shift = {(7,0)}]
         \filldraw (0,0) circle (0.1);
        \filldraw (1,0) circle (0.1);
        \filldraw (2,0) circle (0.1);
        \filldraw (-1,0) circle (0.1);
        \filldraw (0,1) circle (0.1);
        \filldraw (1,1) circle (0.1);
        \filldraw (2,1) circle (0.1);
        \filldraw (-1,1) circle (0.1);
        \filldraw (0,2) circle (0.1);
        \filldraw (1,2) circle (0.1);
        \filldraw (2,2) circle (0.1);
        \filldraw (-1,2) circle (0.1);
        \filldraw (0,-1) circle (0.1);
        \filldraw (1,-1) circle (0.1);
        \filldraw (2,-1) circle (0.1);
        \filldraw (-1,-1) circle (0.1);
        \draw[fill,color=black,opacity=.2] (0.5,-0.5)--(0.5,0.5)--(-0.5,0.5)--(-0.5,1.5)--(-1.5,1.5)--(-1.5,2.5)--(0.5,2.5)--(0.5,1.5)--(1.5,1.5)--(1.5,-0.5)--cycle;
        \draw[fill,color=blue,opacity=.2] (-1.5,0.5)--(-1.5,1.5)--(-0.5,1.5)--(-0.5,0.5)--cycle;
        \node [anchor = north east] at (-1,1) {$c$};
    \end{scope}
        \begin{scope}[shift = {(0,-5)}]
         \filldraw (0,0) circle (0.1);
        \filldraw (1,0) circle (0.1);
        \filldraw (2,0) circle (0.1);
        \filldraw (-1,0) circle (0.1);
        \filldraw (0,1) circle (0.1);
        \filldraw (1,1) circle (0.1);
        \filldraw (2,1) circle (0.1);
        \filldraw (-1,1) circle (0.1);
        \filldraw (0,2) circle (0.1);
        \filldraw (1,2) circle (0.1);
        \filldraw (2,2) circle (0.1);
        \filldraw (-1,2) circle (0.1);
        \filldraw (0,-1) circle (0.1);
        \filldraw (1,-1) circle (0.1);
        \filldraw (2,-1) circle (0.1);
        \filldraw (-1,-1) circle (0.1);
        \draw[fill,color=black,opacity=.2] (0.5,-0.5)--(0.5,0.5)--(-0.5,0.5)--(-0.5,1.5)--(-1.5,1.5)--(-1.5,2.5)--(0.5,2.5)--(0.5,1.5)--(1.5,1.5)--(1.5,-0.5)--cycle;
        \draw[fill,color=blue,opacity=.2] (-0.5,-0.5)--(-0.5,0.5)--(0.5,0.5)--(0.5,-0.5)--cycle;
        \node [anchor = north east] at (0,0) {$c$};
    \end{scope}
    \begin{scope}[shift = {(7,-5)}]
         \filldraw (0,0) circle (0.1);
        \filldraw (1,0) circle (0.1);
        \filldraw (2,0) circle (0.1);
        \filldraw (-1,0) circle (0.1);
        \filldraw (0,1) circle (0.1);
        \filldraw (1,1) circle (0.1);
        \filldraw (2,1) circle (0.1);
        \filldraw (-1,1) circle (0.1);
        \filldraw (0,2) circle (0.1);
        \filldraw (1,2) circle (0.1);
        \filldraw (2,2) circle (0.1);
        \filldraw (-1,2) circle (0.1);
        \filldraw (0,-1) circle (0.1);
        \filldraw (1,-1) circle (0.1);
        \filldraw (2,-1) circle (0.1);
        \filldraw (-1,-1) circle (0.1);
        \draw[fill,color=black,opacity=.2] (0.5,-0.5)--(0.5,0.5)--(-0.5,0.5)--(-0.5,1.5)--(-1.5,1.5)--(-1.5,2.5)--(0.5,2.5)--(0.5,1.5)--(1.5,1.5)--(1.5,-0.5)--cycle;
        \draw[fill,color=blue,opacity=.2] (0.5,-1.5)--(0.5,-0.5)--(2.5,-0.5)--(2.5,-1.5)--cycle;
        \node [anchor = north east] at (1,-1) {$c$};
    \end{scope}
    \end{tikzpicture}
\end{center}

\begin{proof} Let $f: \mathbb{I}_{\langle A,B\rangle}\to \mathbb{I}_{\langle C,D\rangle}$ be a  
nonzero morphism. By Proposition~\ref{prop:hom_spread} and Lemma~\ref{lem:irreducible}, it suffices to consider only the case where $\langle A,B\rangle \neq \langle C,D\rangle$ and $f$ is either injective or surjective.

Suppose first that $f$ is surjective. Then $f_x\neq 0$ for every $x\in \langle C,D\rangle$, so $\langle C,D\rangle\subseteq \langle A,B\rangle$. Then $\langle A,B\rangle$ must include an element $x$ such that $x\not\leq \langle C,D\rangle$. Moreover, we can assume that $\langle C,D\rangle \leq x$ since $\langle A,B\rangle$ is connected. Without loss of generality, suppose $\langle C,D\rangle \cover x$. By convexity, $\langle A,x\rangle\subseteq \langle A,B\rangle$.
Now suppose $\rangle D,x\rangle\not\subseteq \langle A,B\rangle$. .
Then we get a proper factorisation of $f$ of the shape $\mathbb{I}_{\langle A,B\rangle} \xto{h} \mathbb{I}_{\langle C,D\rangle \cup \rangle D,x\rangle} \to \mathbb{I}_{\langle C,D\rangle}$, where $h_y=f_y$ for all $y\in \langle A,B\rangle$. Note that for all $y\in \ \rangle D,x\rangle$, we have $h_y=f_y=0$. Thus if $\langle A,B\rangle \neq \langle C,D\rangle \cup \rangle D,x\rangle$, then $f$ is not irreducible.

Now suppose $\langle A,B\rangle \neq \langle C,D\rangle \cup \rangle D,x\rangle$. We claim that $f$ is $\cX$-irreducible. For a contradiction, suppose note. Then Lemma~\ref{lem:irreducible} implies that we have a proper factorisation of the shape $$\mathbb{I}_{\langle C,D\rangle \cup \rangle D,x\rangle} \xto{g} \mathbb{I}_{\langle E,F\rangle} \xto{h} \mathbb{I}_{\langle C,D\rangle}$$ with $h$ irreducible. As $f$ is surjective, so is $h$. Then $\langle E,F\rangle$ must contain $\langle C,D \rangle$. By the same argument as before, it must also contain $\rangle D, x'\rangle$ for some $\langle C,D\rangle \cover x'$. As $x'$ is maximal in $\langle E,F\rangle$ and $g$ is a morphism, we have that $x'$ must also be in $\langle C,D\rangle \cup \rangle D,x\rangle$. But the only element covering $\langle C,D\rangle$ in $\langle C,D\rangle \cup \rangle D,x\rangle$ is $x$, so $x'=x$. So $\langle E,F\rangle$ contains the support of the left most module. Then $g$ is injective (by Proposition~\ref{prop:hom_spread}) and $\langle E,F\rangle$ must contain some element that is not in $\langle C,D\rangle \cup \rangle D,x\rangle$, say $y$. As before, $y \not\leq \langle C,D\rangle$. Then $h$ has a factorisation of the shape $\mathbb{I}_{\langle E,F\rangle} \to \mathbb{I}_{\langle C,D\rangle \cup \rangle D,x\rangle} \to \mathbb{I}_{\langle C,D\rangle}$, that is proper because the first module is supported at $y$, a contradiction to its $\cX$-irreducibility. 

Now suppose instead that $f$ is injective, so $f_x\neq 0$ for all $x\in \langle A,B\rangle$, which implies $\langle A,B\rangle\subseteq \langle C,D\rangle$. As $f$ is not an isomorphism, there must be some $x\in \langle C,D\rangle \setminus \langle A,B\rangle$, and there must be some $c\in C$ such that $c\leq x$, which implies $c\not \in \langle A,B\rangle$. This implies we get a factorisation of the shape $\mathbb{I}_{\langle A,B\rangle} \to \mathbb{I}_{\langle C,D\rangle \setminus \{c\}} \to \mathbb{I}_{\langle C,D\rangle}$, where the first morphism is injective since $f$ is. Note that while $\langle C,D\rangle\setminus \{c\}$ need not be connected, it still yields a module and a factorisation of the shape above. As $\langle A,B\rangle$ is connected, it is thus contained in one of the connected components of $\langle C,D\rangle$, say $\langle E,F\rangle$. We then get a factorisation of the shape $\mathbb{I}_{\langle A,B\rangle}\to \mathbb{I}_{\langle C,D\rangle\setminus\{c\}}\to \mathbb{I}_{\langle E,F\rangle} \to \mathbb{I}_{\langle C,D\rangle}$. Now as the third module is indecomposable, and the last morphism of that sequence is not invertible, either the composition of the first two morphisms is invertible or $f$ is not $\cX$-irreducible. In the second case we are done, so suppose the composition is invertible; that is, that $\langle A,B\rangle=\langle E,F\rangle$. Now we get a factorisation of the shape $\mathbb{I}_{\langle A,B\rangle} \to \mathbb{I}_{\langle A,B\rangle\cup \langle C,A\langle}\to \mathbb{I}_{\langle C,D\rangle}$. As $c$ is in the support of the second module but not the first, the leftmost morphism cannot be invertible. As above, we are done if the right morphism is not invertible, so suppose that $\langle C,D\rangle=\langle A,B \rangle \cup \langle C,A\langle$. Now if there exist $c\neq c'$ such that $\{c,c'\}\subseteq C$. Then we get a factorisation of the shape $\mathbb{I}_{\langle A,B\rangle} \to \mathbb{I}_{\langle C,D\rangle\setminus \{c\}} \to \mathbb{I}_{\langle C,D\rangle}$. As $c$ is in the support of the second module but not the first and $c'$ is in the support of the third module but not the second, that factorisation is proper. We conclude that if $f$ is $\cX$-irreducible, then $\langle C,A\langle= \langle c,A\langle$ and the condition (2) is verified. 

Now suppose that condition (2) is verified. It remains to show that $f: \mathbb{I}_{\langle A,B\rangle}\to \mathbb{I}_{\langle A,B\rangle\cup \langle c,A\langle}$ is irreducible. Suppose it isn't, so that Lemma~\ref{lem:irreducible} yields a a factorisation of the shape $\mathbb{I}_{\langle A,B\rangle} \xto{g} \mathbb{I}_{\langle E,F \rangle} \xto{h} \mathbb{I}_{\langle A,B\rangle\cup \langle c,A\langle}$ with $g$ irreducible. As $f$ is injective, so is $g$. By what we just showed, $\langle A,B\rangle$ must then be a connected component of $\langle E,F\rangle\setminus \{e\}$ for some $e\in E$ and verify $\langle E,F\rangle=\langle A,B\rangle\cup \langle e,A\langle$. Note that $e<A$. Now as $h$ is a morphism that is non-zero on $\langle A,B\rangle \subseteq \langle E,F\rangle$, 
it must be that $e \in \langle A,B\rangle \cup \langle c,A\langle$. 
As $e<A$, we must then have $c\leq e$. Also, as $\langle A,B\rangle$ is a connected component of $\langle C,D\rangle\setminus \{c\}$, we have $c\cover A$. Since $c\leq e<A$, we get $c=e$, which implies $\langle E,F\rangle=\langle A,B\rangle\cup \langle e,A\langle =\langle A,B\rangle \cup \langle c,A\langle $. This means that $h$ is an isomorphism, a contradiction.\end{proof}

We conclude this section by returning to the notion of $\add(\cX)$-global dimension. As mentioned in the introduction, \cite[Section~8.3]{OS} and \cite[Sections~6 and~7]{BOOS} establish stability results for certain homological invariants. In both papers, the authors consider sets of spread modules $\cX$ for which the $\add(\cX)$-global dimensions of all finite (grid) posets $\cP$ can be uniformly bounded. Example~\ref{ex:spreads_infinite} below shows uses the rewsults of this section to show that such a uniform bound does not exist when $\cX$ is taken to be either the set of all spread modules or the set of single-source spread modules.

\begin{example}\label{ex:spreads_infinite}
    Let $1 < n \in \mathbb{N}$ and $\cP = \langle (1,1), (n,n) \rangle \subseteq \mathbb{N}^2$. Let $\cX_1$ be the set of all spread modules and $\cX_2$ the set of all single-source spread modules. Denote $G_1 = \bigoplus_{X \in \cX_1} X$ and likewise for $G_2$. We note that $G_1$ and $G_2$ are both generator-cogenerators of $\rep \cP$. For each $k \in \{1,2\}$, the functor $\Hom_\cP(-,G_k)$ induces an equivalence of categories between $\add(\cX_k)$ and $\proj(\End_\cP(G_k)^{\mathrm{op}})$ and sends $\cF^{\cX_k}$-exact sequences to exact sequences. See e.g. \cite[Section~II.2]{ARS}.
    
    Consider $\cS,\cT \subseteq \cP$ spreads such that $(1,1) \in \cS$. Then $\cS$ is a single-source spread and, by Proposition~\ref{prop:hom_spread}, $\Hom_\cP(\mathbb{I}_\cS,\mathbb{I}_\cT) \neq 0$ if and only if $(1,1) \in \cT$ and $\cT \subseteq \cS$ (in which case $\cT$ is also a single-source spread). Moreover, the same proposition implies that if $\Hom_\cP(\mathbb{I}_\cS, \mathbb{I}_\cT) \neq 0$, then $\dim_\field\Hom_\cP(\mathbb{I}_\cS,\mathbb{I}_\cT) = 1$ and there is a surjective map $q_\cS^\cT: \mathbb{I}_\cS \rightarrow \mathbb{I}_\cT$ which satisfies $q_\cS^\cT((1,1)) = 1_\field$. By Propositions~\ref{prop:irreducibles_single_source} and~\ref{prop:irreduciblesbetweenspreads}, we have that $q_\cS^\cT$ is $\cX_1$-irreducible if and only if $q_\cS^\cT$ is $\cX_2$-irrecucible if and only if $\cT = \cS \cup \{x\}$ for some $\cS \cover x$.
    
    Denote $x_i = (i, n+1-i)$ for $1 \leq i \leq n$. Let $B = \{x_i\mid i \in \{1,\ldots,n\}\}$, and consider the spread $\cS = \langle \{(1,1)\}, B\rangle$. We claim that the simple top of the right $\End_\cP(G_k)^{\mathrm{op}}$-module $\Hom_\cP(\mathbb{I}_\cS,G_k)$ has projective dimension $n$ for $k \in \{1,2\}$. As a consequence, Proposition~\ref{prop:generator_cogenerator} will imply that $\dim_{\cX_k}(\field_\cP) \geq n-2$ for $k \in \{1,2\}$.

    We prove the claim using a construction similar to the Koszul complexes appearing in \cite{CGRST}. Namely, we define a cochain complex $U^\bullet$ as follows. For $p \in \mathbb{Z}$, let
    $$U^p = \bigoplus_{B' \subseteq B: |B'| = p} \mathbb{I}_{\cS \setminus B'}.$$
    For $B' \subseteq B$, write $B' = \{x_{i_1}, \ldots, x_{i_p}\}$ with $i_1 < \cdots < i_p$. Then, for $j \in \{1,\ldots,p\}$, the differential of $U^\bullet$ is given by $(-1)^{j+1}q_{(\cS \setminus B') \cup x_{i_j}}^{\cS \setminus B'}$ when restricted and corestricted to the appropriate summands. For example, take $n = 3$. Expressing each spread as a matrix of 0s and 1s,  identifying each spread $\cT$ with the spread module $\mathbb{I}_\cT$, and writing $q$ in place of each $q_{(\cS\setminus B')\cup x_{i_j}}^{\cS \setminus B'}$, we get
    $$
        U^\bullet =  \hspace{0.2cm} 0 \rightarrow  {\tiny \begin{pmatrix}1 \ 0 \ 0\\1\ 1\ 0\\1\ 1\ 1\end{pmatrix}} \xrightarrow{d^0} {\tiny \begin{pmatrix}0\ 0\ 0\\1\ 1\ 0\\1\ 1\ 1\end{pmatrix}}\oplus {\tiny \begin{pmatrix}1\ 0\ 0\\1\ 0\ 0\\1\ 1\ 1\end{pmatrix}}\oplus {\tiny \begin{pmatrix}1\ 0\ 0\\1\ 1\ 0\\1\ 1\ 0\end{pmatrix}} \xrightarrow{d^1}{\tiny\begin{pmatrix}0\ 0\ 0\\1\ 0\ 0\\1\ 1\ 1\end{pmatrix}}\oplus {\tiny \begin{pmatrix}1\ 0\ 0\\1\ 0\ 0\\1\ 1\ 0\end{pmatrix}}\oplus {\tiny \begin{pmatrix}0\ 0\ 0\\1\ 1\ 0\\1\ 1\ 0\end{pmatrix}} \xrightarrow{d^2} {\tiny \begin{pmatrix}0\ 0\ 0\\1\ 0\ 0\\1\ 1\ 0\end{pmatrix}} \rightarrow 0,
    $$
    with
    $$d^0 = {\scriptsize\begin{bmatrix} q\\-q\\q\end{bmatrix}}, \qquad\qquad d^1 = {\scriptsize\begin{bmatrix} q & q & 0\\0 & -q & -q\\ -q & 0 & q\end{bmatrix}}, \qquad\qquad d^2 = {\scriptsize\begin{bmatrix} q & q & q\end{bmatrix}}.$$
    It is straightforward to verify that $U^\bullet$ is exact (for any $n$). Moreover, the last sentence of the second paragraph of this example implies that $U^\bullet$ is also both $\cF^{\cX_1}$- and $\cF^{\cX_2}$-exact. For each $k \in \{1,2\}$, applying the functor $\Hom_\cP(-,G_k)$ to $U^\bullet$ then yields a minimal projective resolution of the simple top of $\Hom_\cP(\mathbb{I}_\cS,G_k)$. This resolution contains $\Hom_\cP(\mathbb{I}_{\cS \setminus B},G_k)$ in degree $n$, verifying the claim.
\end{example}

\begin{remark}
    For $U^p$ in Example~\ref{ex:spreads_infinite}, the direct summands of $U^p$ form a graded lattice whose atoms can be identified with irreducible morphisms from $\mathbb{I}_\cS$. We note that, while this behavior does not seem uncommon, it is not a general phenomenon. It should not be expected while computing spread resolutions over other posets, or when changing the family $\cX$, as one can build counterexamples in both cases.
\end{remark}


\section{Exact structures for infinite posets}\label{sec:infinite}

Let $\cP = \cT_1\times \cdots \times \cT_n$ be a product of finitely many totally ordered sets, equipped the usual product order. We assume that $|\cT_i| \geq 2$ for all $i$. The purpose of this section is to explain how the theory of exact structures over finite posets can be extended to include finitely presentable modules over an infinite such $\cP$.

We denote by $\Pi_i: \cP \rightarrow \cT_i$ the projection map for each $i$. Then $\cP$ is a (join semi-)\\lattice with the join of a finite subset $\cQ \subseteq \cP$ given by
$\bigvee \cQ = \left(\max\left\{\Pi_i(x) \mid x \in \cQ\right\}\right)_i.$
We will be interested in certain distinguished subsets of $\cP$ characterized as follows.

\begin{definition}\label{def:grid}
    \begin{enumerate}
    	\item Let $\cT'_i \subseteq \cT_i$ be a finite subset for each $i$. We call $\cQ= \cT_1'\times \cdots \times \cT'_n \subseteq \cP$ a \emph{finite aligned subgrid} of $\cP$. We denote by $\fG(\cP)$ the set of finite aligned subgrids of $\cP$.
	\item Let $\cQ \subseteq \cP$ be a finite subset. We denote by $\cG(\cQ)$ the smallest finite aligned subgrid of $\cP$ containing $\cQ$.
	\end{enumerate}
\end{definition}

We note that each finite aligned subgrid of $\cP$ is also a sublattice of $\cP$. Moreover, we can realize the incidence algebra of a finite aligned subgrid in the following way.

\begin{remark}\label{rem:endomorphism_incidence}
	Let $\cQ \in \fG(\cP)$ be a finite aligned subgrid, and let $e_\cQ= \sum_{x \in \cQ} \langle x,x\rangle \in \field_\cP$. Then 
	$e_\cQ \cdot \field_\cP \cdot e_\cQ \cong \End_{\field_\cP}(e_\cQ\cdot \field_\cP)^{\mathrm{op}} \cong \field_\cQ.$
\end{remark}

For use later in this section, we consider the following definition.

\begin{definition}\label{def:grid_of_module}
	\begin{enumerate}
\item Let $M \in \rep \cP$, and let $r: P_1 \rightarrow P_0$ be a minimal projective presentation of $M$\footnote{By definition, this means that $P_1 \xrightarrow{r} P_0$ can be completed to a minimal $\proj(\rep \cP)$-resolution of $M$.}. We denote
	$$\cL(M) = \cG \{x \in \cP \mid \langle x,x\rangle\cdot \field_\cP \text{ is a direct summand of } P_1 \oplus P_0\}.$$
		\item For $\cQ \in \fG(\cP)$, we denote by $\rep(\cP; \cQ)$ the set full subcategory of $\rep \cP$ consisting of representations $M$ which satisfy $\cL(M) \subseteq \cQ$.
	\end{enumerate}
\end{definition}

The fact that $\cL(M)$ is well-defined follows from the well-known fact that minimal projective presentations are unique up to isomorphism.


\subsection{The restriction and extension functors}\label{subsec:restrict_extend}

In this section, we define an adjunction between (a subcategory of) $\rep \cP$ and $\rep \cQ$ for $\cQ$ a finite aligned subgrid of $\cP$. We first consider the following definitions.

\begin{definition}\label{def:floor}
Let $\cQ \in \fG(\cP)$ be a finite aligned subgrid of $\cP$.
    \begin{enumerate}
        \item Denote $\cQ^+ = \{x \in \cP \mid \exists y \in \cQ: y \leq x\}$.
        \item For $x \in \cQ^+$, denote $\lfloor x\rfloor_\cQ = \max\{y \in \cQ \mid y \leq x\} = \bigvee\{y \in \cQ \mid y \leq x\}$.
        \item For $y \in \cQ$, denote $\lceil y\rceil_\cQ = \{x \in \cQ^+ \mid \lfloor x\rfloor_\cQ = y\}$.
        \item Denote $\rep(\cP; \cQ^+)$ the full subcategory of $\rep \cP$ consisting of representations with support contained in $\cQ^+$. Alternatively, on the level of objects we have $$\rep(\cP; \cQ^+) = \bigcup_{\cQ' \in \fG(\cQ^+)} \rep(\cP; \cQ').$$
    \end{enumerate}
\end{definition}

\begin{remark}\label{rem:restriction}
	For $\cQ \in \fG(\cP)$, we note the following.
	\begin{enumerate}
		\item The category $\rep(\cP; \cQ^+)$ is closed under taking extensions, subrepresentations, and quotient representations in $\rep \cP$.
		\item Let $x \in \cQ^+$ and $y \in \cQ$. Then $x \in \lceil \lfloor x\rfloor_\cQ\rceil_\cQ$ and $y = \lfloor z\rfloor_\cQ$ for all $z \in \lceil y\rceil_\cQ$.
		\item Let $x \in \cQ^+$. Then $\lfloor x\rfloor_\cQ$ is the unique element of $\cQ$ which satisfies
		$$\Pi_i(\lfloor x\rfloor_\cQ) = \max\left\{\Pi_i(y) \mid y \in \cQ \text{ and } y \leq x\right\}$$
		for all $i$.
	\end{enumerate}

\end{remark}

The following properties of finite aligned subgrids, neither of which hold for arbitrary finite sublattices, will be critical.

\begin{lemma}\label{lem:aligned}
    Let $\cQ \in \fG(\cP)$, and let $y \in \cQ$. Then
    \begin{enumerate}
	\item For all $x \in \lceil y\rceil_\cQ$ and $y \leq y' \in \cQ$, one has $y' \vee x \in \lceil y'\rceil_\cQ$. In particular, there exists $x' \in \lceil y'\rceil_\cQ$ such that $x \leq x'$.
	\item $\lceil y\rceil_\cQ$ is a sublattice of $\cP$.
    \end{enumerate}
\end{lemma}

\begin{proof}
    (1) Given $y \leq y' \in \cQ$ and $x \in \lceil y\rceil_\cQ$, set $x' = x \vee y'$. Then for each $i$, we have $\Pi_i(y) \leq \{\Pi_i(x), \Pi_i(y')\} \leq \Pi_i(\lfloor x'\rfloor_\cQ) \leq \Pi_i(x').$
    Now let $z \in \cP$ such that
    $$\Pi_i(z) = \begin{cases} \Pi_i(y') & \Pi_i(y') \leq \Pi_i(x)\\ \Pi_i(y) & \Pi_i(y') > \Pi_i(x).\end{cases}$$
    Then by construction $z \in \cQ$ and $y \leq z \leq x$, and so $y = z$. Now if $\Pi_i(y') > \Pi_i(x)$, then $\Pi_i(y') = \Pi(x')$ and so $\Pi_i(y') = \Pi_i(\lfloor x'\rfloor_\cQ)$. Otherwise, we have $\Pi_i(y) = \Pi_i(y') \leq \Pi_i(\lfloor x'\rfloor_\cQ) \leq \Pi_i(x') = \Pi_i(x)$. Since $\cQ$ is an aligned grid, it follows that $\Pi_i(\lfloor x'\rfloor_\cQ) = \Pi_i(y)$. We conclude that $y = \lfloor x'\rfloor_\cQ$.
    
        (2) This follows immediately from the characterization in Remark~\ref{rem:restriction}(3).
\end{proof}

\begin{example}\label{ex:not_complete}
	\begin{enumerate}
		\item Let $\cP = \{1,2,3\}\times \{1,2,3\}$ with the usual product order, and let $\cQ = \{(1,1), (2,2), (3,2), (2,3), (3,3)\}$. Let $y = (1,1)$, $y' = (2,2)$. Then $\lceil y\rceil_\cQ = \Pi_1^{-1}(1) \cup \Pi_2^{-1}(1)$ is not a sublattice of $\cP$. Furthermore, $(1,3) \in \lceil y\rceil_\cQ$, but there does not exist $x' \in \lceil y'\rceil_\cQ$ such that $(1,3) \leq x'$. This shows that neither property (1) nor property (2) from Lemma~\ref{lem:aligned} hold for arbitrary sublattices of $\cP$.
		\item If $\cP$ is a complete lattice and $\cQ \in \fG(\cP)$, then $\lceil y\rceil_\cQ$ is often not a complete sublattice of $\cP$. For example, take $\cP = [0,1] \subseteq \mathbb{R}$, $\cQ = \{0,1\}$, and $y = 0$.
	\end{enumerate}
\end{example}

We are now ready to describe our first functors.

\begin{definition}\label{def:restriction}
	Let $\cQ \in \fG(\cP)$. We define the \emph{restriction functor} $(-)|_\cQ: \rep \cP \rightarrow \rep \cQ$ to be the functor given by precomposing with the inclusion map $\cQ \hookrightarrow \cP$. Alternatively, by identifying $\rep\cP$ with $\mods \field_\cP$, Remark~\ref{rem:endomorphism_incidence} gives a natural isomorphism	$$(-)|_\cQ \cong \Hom_{\field_\cP}(e_\cQ \cdot \field_\cP,-): \mods \field_\cP \rightarrow \mods (e_\cQ \cdot \field_\cP\cdot e_\cQ).$$
\end{definition}

It follows from its definition that $(-)|_\cQ$ is an exact functor. Moreover, this functor admits a left adjoint $(-) \otimes_{(e_\cQ\cdot \field_\cP\cdot e_\cQ)}(e_\cQ \cdot \field_\cP)$. Our next goal is to give an explicit description of this tensor product functor and to show that it is also exact. Towards that end, we consider the following.

\begin{definition}\label{def:extension_restriction}
Let $\cQ \in \fG(\cP)$. We define the \emph{extension functor}
$\lceil -\rceil_\cQ: \rep \cQ \rightarrow \Rep \cP$ as follows. For $N \in \rep \cQ$ and $x \leq x' \in \cP$, denote
	\begin{align*} \lceil N\rceil_\cQ(x) &= \begin{cases} N(\lfloor x\rfloor_\cQ) & x \in \cQ^+\\ 0 & \text{otherwise}.\end{cases}\\
	\lceil N\rceil_\cQ(x,x') &= \begin{cases} N(\lfloor x\rfloor_\cQ, \lfloor x\rfloor_\cQ) & x, x' \in \cQ^+\\ 0 & \text{otherwise}.\end{cases}\end{align*}
\end{definition}

We observe the following facts about the extension functor.

\begin{lemma}\label{lem:extension}
	Let $\cQ \in \fG(\cP)$. Then:
	\begin{enumerate}
		\item The functor $\lceil -\rceil_\cQ$ is exact.
		\item For all $y \in \cQ$, one has $\lceil \langle y,y\rangle\cdot \field_\cQ\rceil_\cQ = \langle y,y\rangle\cdot \field_\cP$. In particular, $\lceil -\rceil_\cQ$ preserves projectives.
		\item The functor $\lceil -\rceil_\cQ$ preserves projective presentations. In particular, the essential image of $\lceil -\rceil_\cQ$ is contained in $\rep(\cP; \cQ)$.
	\end{enumerate}
\end{lemma}

\begin{proof}
	Items (1) and (2) follow immediately from the definition. Item (3) is then a consequence of (1) and (2).
\end{proof}

\begin{remark}\label{rem:kan}
    It is observed in \cite[Section~4.4]{BOO} that $\lceil -\rceil_\cQ$ coincides with the functor sending each representation $N \in \rep \cP$ to its left Kan extension $\mathsf{Lan}(N)$ along the inclusion functor $\cQ \hookrightarrow \cP$. See also \cite[Section~8]{CJT}. We note that the functor $\lceil -\rceil_\cQ$ also occurs under the name ``induction functor'' in e.g. \cite[Section~5.3]{simson}. 
\end{remark}

We are now ready to prove the following.

\begin{proposition}\label{prop:prop_restrict_1}
	Let $\cQ \in \fG(\cP)$. Then:
	\begin{enumerate}
		\item $\mathrm{Id}_{\rep \cQ} = (-)|_\cQ \circ \lceil-\rceil_\cQ$.
		\item There is a natural transformation $\varepsilon: \mathrm{Id}_{\rep \cP} \rightarrow \lceil-\rceil_\cQ \circ (-)|_\cQ$ given by
		$$\varepsilon_M(x) = \begin{cases} M(\lfloor x\rfloor_\cQ,x) & x \in \cQ^+\\ 0 & \text{otherwise}.\end{cases}$$
		\item There is an adjunction
	$\lceil -\rceil_\cQ: \rep \cQ \rightleftarrows \rep \cP: (-)|_\cQ.$ In particular, $\lceil -\rceil_\cQ$ is naturally isomorphic to $(-) \otimes_{(e_\cQ\cdot \field_\cP\cdot e_\cQ)}(e_\cQ \cdot \field_\cP)$.
	\end{enumerate}
\end{proposition}

\begin{proof}
	Items (1) and (2) follow immediately from the definitions. The equality in (1) and the natural transformation $\varepsilon$ then form the unit and counit of the desired adjunction. The natural isomorphism between the extension functor and the tensor functor then follows from the natural isomorphism between the restriction functor and the Hom-functor observed in Definition~\ref{def:restriction}.
\end{proof}

We conclude this section with the following, which can be compared with e.g. \cite[Proposition~2.9]{LW} and \cite[Lemmas~5 and~6]{BS}.

\begin{lemma}\label{lem:restrict_projective}
	Let $\cQ \in \fG(\cP)$. Then there are inverse equivalences of categories
	$$\lceil-\rceil_\cQ: \rep \cQ \rightleftarrows \rep(\cP; \cQ): (-)|_\cQ.$$
	In particular, $\lceil-\rceil_\cQ \circ (-)|_\cQ$ is naturally isomorphic to the identity on $\rep(\cP; \cQ)$.
\end{lemma}

\begin{proof}
	Recall that both functors are exact. Analogously to Lemma~\ref{lem:extension}, it is readily verified that the functor $(-)|_\cQ$ preserves the projective presentations of representations in $\rep(\cP;\cQ)$. The result then follows.
\end{proof}

In some sense, Lemma~\ref{lem:restrict_projective} morally says that every finitely presentable representation of $\cP$ can be identified with a representation of a finite poset. (This is related to the notion of ``tameness'', for which we refer to \cite{miller} and \cite[Section~8]{CJT}.) We note, however, that the adjunction in Proposition~\ref{prop:prop_restrict_1} is in general not bi-directional. Indeed, we will see in Section~\ref{sec:infinite_exact_structure} that having a \emph{left} adjoint to the extension functor is useful in extending results about exact structures from finite posets to infinite ones.


\subsection{The contraction functor}\label{sec:contract}

In this section, we define a \emph{left} adjoint to the extension functor.

\begin{definition}\label{def:contraction}
    Let $\cQ \in \fG(\cP)$ be a finite aligned subgrid. We define the \emph{contraction functor} $\lfloor-\rfloor_\cQ: \rep(\cP; \cQ^+) \rightarrow \rep \cQ$ as follows. For $M \in \rep(\cP;\cQ^+)$, and $y \in \cQ$, we first define a diagram $\lceil y\rceil_\cQ(M)$ such that:
       \begin{itemize}
           \item The set of vertices of $\lceil y\rceil_\cQ(M)$ is $\{M(x) \mid x \in \lceil y\rceil_\cQ\}$
           \item The set of arrows of $\lceil y\rceil_\cQ(M)$ is $\{M(x,x') \mid x \leq x' \in \lceil y\rceil_\cQ\}$.
       \end{itemize}
       Note that the diagram $\lceil y\rceil_\cQ(M)$ is directed by Lemma~\ref{lem:aligned}(2). Moreover, for any choice of $\cQ' \in \fG(\cP)$ which satisfies $\cQ \cup \cL(M) \subseteq \cQ'$, Lemma~\ref{lem:restrict_projective} implies that the filtered colimit of $\lceil y\rceil_\cQ(M)$ coincides with the filtered colimit of the full subdiagram supported on the vertices $M(x)$ for $x \in \lceil y\rceil_\cQ \cap \cQ'$. We denote by $\lfloor M\rfloor_\cQ(y)$ the filtered colimit of $\lceil y\rceil_\cQ(M)$ and note that $\dim_\field \lfloor M\rfloor_\cQ(y) < \infty$.

       For $y \leq y' \in \cQ$, Lemma~\ref{lem:aligned}(1) implies that there is an induced map $\lfloor M\rfloor_\cQ(y,y'): \lfloor M\rfloor_\cQ(y) \rightarrow \lfloor M\rfloor_\cQ(y')$ coming from the structure maps $M(x,y' \vee x)$ for $x \in \lceil y\rceil_\cQ$. Moreover, this construction is functorial in $M$ and thus yields the desired functor $\lfloor-\rfloor_\cQ$.
\end{definition}

\begin{example}\label{ex:contraction}
    Let $\cP = \mathbb{R}^2$. To help distinguish between vector spaces of the form $M(x)$ and linear maps of the form $M(x,x')$, we denote by $(ab) \in \mathbb{R}^n$ the point with $\Pi_1(ab) = a$ and $\Pi_2(ab) = b$. With this convention, let $\cQ = \{(00), (20), (02), (22)\}$ and let $M = \mathbb{I}_{\langle \{(01),(10)\},(21)\langle}$. See Figure~\ref{fig:contraction} for an illustration.

    The colimits defining the vector spaces of $\lfloor M\rfloor_\cQ$ can explicitly be realized as follows:
    \begin{align*}
        \lfloor M\rfloor_\cQ(00) &= \colim_{0 \leq a < 2, 0 \leq b < 2} M(ab) = M(11) = K\\
        \lfloor M\rfloor_\cQ(20) &= \colim_{2 \leq a, 0 \leq b < 2} M(ab) = M(21) = 0\\
        \lfloor M\rfloor_\cQ(02) &= \colim_{0 \leq a < 2, 2 \leq b} M(ab) = M(12) = K\\
        \lfloor M\rfloor_\cQ(22) &= \colim_{2 \leq a, 2 \leq b} M(ab) = M(22) = 0\\
    \end{align*}
    Thus we have $\lfloor M\rfloor_\cQ = \mathbb{I}_{\langle (00), (20)\rangle} \in \rep(\cQ)$. We also compute $M|_\cQ = \mathbb{I}_{\langle (20),(20)\rangle} \oplus \mathbb{I}_{\langle(02),(02)\rangle} \in \rep(\cQ)$, showing that the functors $\lfloor-\rfloor_\cQ$ and $(-)|_\cQ$ generally do not coincide. 
\end{example}

\begin{figure}
    \begin{tikzpicture}[scale = 1.2]
        \filldraw (0,0) circle (0.1);
        \filldraw (1,0) circle (0.1);
        \filldraw (2,0) circle (0.1);
        \filldraw (0,1) circle (0.1);
        \filldraw (1,1) circle (0.1);
        \filldraw (2,1) circle (0.1);
        \filldraw (0,2) circle (0.1);
        \filldraw (1,2) circle (0.1);
        \filldraw (2,2) circle (0.1);
        \node[anchor=north west] at (0,0) {(00)};
        \node[anchor=north west] at (1,0) {(10)};
        \node[anchor=north west] at (2,0) {(20)};
        \node[anchor=north west] at (0,1) {(01)};
        \node[anchor=north west] at (1,1) {(11)};
        \node[anchor=north west] at (2,1) {(21)};
        \node[anchor=north west] at (0,2) {(02)};
        \node[anchor=north west] at (1,2) {(12)};
        \node[anchor=north west] at (2,2) {(22)};
        \draw[fill,color=black,opacity=.2] (1,0)--(3,0)--(3,1)--(2,1)--(2,3)--(0,3)--(0,1)--(1,1)--cycle;

        \draw[very thick] (0,3)--(0,1)--(1,1)--(1,0)--(3,0);
        \draw[very thick,dashed] (2,3)--(2,1)--(3,1);
    \end{tikzpicture}
    \caption{The module $M = \mathbb{I}_{\langle \{(01),(10)\},(21)\langle} \in \rep(\mathbb{R}^2)$ from Example~\ref{ex:contraction}.}\label{fig:contraction}
\end{figure}

\begin{remark}\label{rem:contraction_versus_restriction}
    While the above example show that the functors $\lfloor-\rfloor_\cQ$ and $(-)|_\cQ$ do not coincide in general, they do coincide on the subcategory $\rep(\cP;\cQ)$.
\end{remark}

The following is an immediate consequence of the fact that filtered colimits commute with finite limits in any abelian category.

\begin{lemma}\label{lem:exact_contraction}
    Let $\cQ \in \fG(\cP)$. Then the functor $\lfloor-\rfloor_\cQ$ is exact.
\end{lemma}

We now prove that the contraction functor is the left adjoint we want.

\begin{proposition}\label{prop:prop_contract_1}
	Let $\cQ \in \fG(\cP)$. Then:
	\begin{enumerate}
            \item There is a natural transformation $\eta:\lceil-\rceil_\cQ\circ  \lfloor-\rfloor_\cQ \rightarrow \mathrm{Id}_{\rep(\cP;\cQ^+)}$ given by
            $\eta_M(x) = \colim_{x \leq z \in \lceil \lfloor x\rfloor_\cQ\rceil_\cQ}M(x,z).$
		\item $\lfloor-\rfloor_\cQ \circ \lceil-\rceil_\cQ = \mathrm{Id}_{\rep \cQ}$.
		\item There is an adjunction
	$\lfloor -\rfloor_\cQ: \rep(\cP;\cQ^+) \rightleftarrows \rep \cQ: \lceil-\rceil_\cQ.$
	\end{enumerate}
\end{proposition}

\begin{proof}
    Item (1) is a consequence of the fact that $\lceil\lfloor M\rfloor_\cQ\rceil_{\cQ}(x) = \colim_{z \in \lceil \lfloor x\rfloor_\cQ\rceil_\cQ}M(z)$
    and that this is a filtered colimit by Lemma~\ref{lem:aligned}(2). Item (2) is a consequence of Lemma~\ref{lem:extension}(3), Remark~\ref{rem:contraction_versus_restriction}, and Proposition~\ref{prop:prop_restrict_1}(1). The natural transformation $\eta$ and the equality in (2) then form the unit and counit of the desired adjunction.
\end{proof}

We also record the following, which follows from direct computation.

\begin{proposition}\label{prop:contract_proj}
    Let $\cQ \in \fG(\cP)$ and $x \in \cQ^+$. Then $\lfloor \langle x,x\rangle\cdot\field_\cP\rfloor_\cQ = \langle\lfloor x\rfloor_\cQ,\lfloor x\rfloor_\cQ\rangle\cdot \field_\cQ$. In particular, the contraction functor $\lfloor-\rfloor_\cQ$ preserves projectives.
\end{proposition}


\subsection{Extensions and contractions of exact structures}\label{sec:infinite_exact_structure}

In this section, we use adjunction to relate exact structures in $\rep \cP$ with those over its finite aligned subgrids. We conclude with several examples. Our arguments utilise the following definitions.

\begin{definition}\label{def:grid_covering}
    Let $\fH \subseteq \fG(\cP)$. We say that $\fH$ is a \emph{grid covering} of $\cP$ if for all $x \in \cP$ there exists $\cQ \in \fH$ such that $x \in \cQ$.
\end{definition}

\begin{definition}\label{def:extended_class}
    Let $\fH \subseteq \fG(\cP)$ be a grid covering. Let $\cX \subseteq \ind(\rep\cP)$ and for all $\cQ \in \fH$ let $\cX_{\cQ} \subseteq \ind(\rep\cQ)$. We say that $(\cX, \{\cX_\cQ\}_{\cQ \in \fH})$ is a \emph{$(\fH,\cP)$-extended projective class} if all of the following hold:
    \begin{enumerate}
        \item $\ind(\proj(\rep \cP)) \subseteq \cX$.
        \item $\ind(\proj(\rep \cQ))\subseteq \cX_\cQ$ for all $\cQ \in \fH$.
        \item $\add(\cX) = \add\left(\bigcup_{\cQ \in \fH} \left\{\lceil Y\rceil_\cQ\mid Y \in \cX_\cQ\right\}\right).$
        \item $\add(\cX_\cQ) = \add \{\lfloor X\rfloor_\cQ \mid X \in \cX \cap \rep(\cP;\cQ^+)\}$ for all $\cQ \in \fH$.
        \item $\Hom_\cP(X, \lceil N\rceil_\cQ) = 0$ for all $\cQ \in \fH$, $N \in \rep \cQ$, and $X \in \cX \setminus \rep(\cP;\cQ^+)$.
        \item For all $M \in \rep \cP$ there exists $\cQ \in \fH$ such that $\cL(M) \subseteq \cQ$.
    \end{enumerate}
\end{definition}

See Example~\ref{ex:upsets} for a detailed discussion of why condition (5) is included in Definition~\ref{def:extended_class} above.

We now prove our main result, which in particular allows Corollary~\ref{cor:AS_1} to be extended to certain classes of representations of infinite posets.

\begin{theorem}\label{thm:extended_class}
    Let $\fH \subseteq \fG(\cP)$ be a grid covering and let $(\cX, \{\cX_{\cQ}\}_\cQ)$ be a $(\fH,\cP)$-extended projective class. Then the following are equivalent.
    \begin{enumerate}
        \item $\proj(\cF_\cX) = \add(\cX)$ and $\rep(\cP)$ has enough $\cF_\cX$-projectives.
        \item $\proj(\cF_{\cX_\cQ}) = \add(\cX_\cQ)$ and $\rep(\cQ)$ has enough $\cF_{\cX_\cQ}$-projectives for all $\cQ \in \fH$.
    \end{enumerate}
\end{theorem}

\begin{proof}
    $(1\implies 2)$: Suppose (1) holds and let $\cQ \in \fH$. Since $\field_\cQ$ is a finite-dimensional algebra, Definition~\ref{def:extended_class}(2) implies that we need only show $\cF_{\cX_\cQ}$ has enough projectives. To see this, let $N \in \rep \cQ$ and let $q: R \rightarrow \lceil N\rceil_\cQ$ be an $\add(\cX)$-cover. Note that $R \in \rep(\cP;\cQ^+)$ by Definition~\ref{def:extended_class}(5), and so $\lfloor R\rfloor_\cQ \in \add(\cX_\cQ)$ by Definition~\ref{def:extended_class}(4). Thus by Lemma~\ref{lem:exact_contraction}, Proposition~\ref{prop:prop_contract_1}, and Definition~\ref{def:extended_class}(3), it follows that $\lfloor q\rfloor_\cQ: \lfloor R\rfloor_\cQ \rightarrow N$ is an $\add(\cX_\cQ)$-cover. We conclude that $\cF_{\cX_\cQ}$ has enough projectives.

    $(2\implies 1)$: Suppose (2) holds. We first show that $\proj(\cF_\cX) = \add(\cX)$. By Definition~\ref{def:extended_class}(1), it suffices to show that if $Z \in \ind(\rep \cP) \setminus \cX$, then $Z \notin \proj(\cF_\cX)$. Consider such a $Z$. By Definition~\ref{def:extended_class}(6), there exists $\cQ \in \fH$ such that $\cL(Z) \subseteq \cQ$. Thus $\lceil \lfloor Z\rfloor_\cQ\rceil_\cQ \cong Z$ by Remark~\ref{rem:contraction_versus_restriction} and Lemma~\ref{lem:restrict_projective}. Moreover, we have that $Z \notin \proj(\cF_{\cX_\cQ})$ by Definition~\ref{def:extended_class}(4).

    Now let $\Delta = (0 \rightarrow \tau \lfloor Z\rfloor_\cQ \rightarrow E \rightarrow \lfloor Z\rfloor_\cQ \rightarrow 0)$
    be the Auslander-Reiten sequence in $\rep \cQ$ which ends in $\lfloor Z\rfloor_\cQ$. Then $\Delta \in \cF_{\cX_\cQ}$ by the proof of \cite[Proposition~5.10]{AS}\footnote{Since $\lfloor Z\rfloor_\cQ \notin \proj(\cF_{\cX_\cQ})$, the relative version of a standard homological algebra fact says there exists some nonsplit short exact sequence ending in $\lfloor Z\rfloor_\cQ$ which lies in $\cF_{\cX_\cQ}$. The remainder of the proof works with any choice of such a sequence taken in place of $\Delta$.}. Definition~\ref{def:extended_class}(5) and Proposition~\ref{prop:prop_contract_1} then imply that $\lceil \Delta\rceil_\cQ \in \cF_{\cX}$. Since the last term of this sequence is $Z$, we conclude that $Z \notin \proj(\cF_\cX)$.

    The argument that $\cF_\cX$ has enough projectives is completely analogous; that is, one starts with an $\add(\cX_\cQ)$-precover of $\lfloor M\rfloor_\cQ$ (where $\cQ$ is chosen such that $\cL(M) \subseteq \cQ$) and applies the functor $\lceil-\rceil_\cQ$.
\end{proof}

The arguments used to prove Theorem~\ref{thm:extended_class} also imply the following.

\begin{corollary}\label{cor:extended_class}
    Let $\fH \subseteq \fG(\cP)$ be a grid covering and let $(\cX, \{\cX_{\cQ}\}_\cQ)$ be a $(\fH,\cP)$-extended projective class. Suppose the two equivalent conditions in Theorem~\ref{thm:extended_class} both hold. Then the following also hold.
    \begin{enumerate}
        \item For all $\cQ \in \fH$ the functors $\lceil-\rceil_\cQ$ and $\lfloor-\rfloor_\cQ$ send $\cF_{\cX_\cQ}$-exact sequences in $\rep(\cQ)$ to $\cF_{\cX}$-exact sequences in $\rep(\cP,\cQ^+)$ and vice versa.
        \item Let $M \in \rep \cP$ and let $\cQ \in \fH$ such that $\cL(M) \subseteq \cQ$. Let $V_\bullet$ be a (minimal) right $\add(\cX_\cQ)$-resolution of $\lfloor M\rfloor_\cQ$. Then $\lceil V_\bullet\rceil_\cQ$ is a (minimal) right $\add(\cX)$-resolution of $M$.
        \item $\dim_\cX(\field_\cP) = \sup_{\cQ \in \fH} \dim_{\cX_\cQ}(\field_\cQ)$.
    \end{enumerate}
\end{corollary}

\begin{example}\label{ex:infinite_exact}
    Let $M \in \rep(\cP)$ and let $\mathbb{I}_{\cS}$ be a single-source fp-spread module such that $\Hom_\cP(\mathbb{I}_{\cS},M) \neq 0$. Then for any finite aligned subgrid $\cQ \in \fG(\cP)$ such that $\cL(M) \subseteq \cQ$, one has $\cS \subseteq \cQ^+$. Moreover, we have the following for an arbitrary grid covering $\fH \subseteq \fG(\cP)$.
    \begin{enumerate}
        \item Take $\cX$ and each $\cX_\cQ$ to be the union of the indecomposable projectives and the hook modules. Then $(\cX, \{\cX_\cQ\}_{\cQ \in \fH})$ is an $(\fH,\cP)$-extended projective class. Thus the results of this section provide an alternative proof of many of the results of \cite[Sections~4.1 and~4.4]{BOO}.
        \item Take $\cX$ and each $\cX_\cQ$ to be the union of the indecomposable projectives and the fp-segment modules. Then $(\cX, \{\cX_\cQ\}_{\cQ \in \fH})$ is an $(\fH,\cP)$-extended projective class. Theorem~\ref{thm:extended_class} and the discussion following Remark~\ref{rem:grid} then imply that $\dim_\cX(\field_\cP) = 2n-2$.
        \item Take $\cX$ and each $\cX_\cQ$ to be the set of single-source fp-spread modules. Then $(\cX, \{\cX_\cQ\}_{\cQ \in \fH})$ is an $(\fH,\cP)$-extended projective class. For $n = 2$, it follows from \cite[Theorem~6.2(1)]{BBH} and Example~\ref{ex:spreads_infinite} that the $\add(\cX)$-global dimension of $\field_\cP$ is infinite, but not realisably infinite.
    \end{enumerate}
\end{example}

\begin{example}\label{ex:infinite_exact_2}
    Suppose $\cP$ has a unique minimum element $m$, and let $\fH \subseteq \fG(\cP)$ be a grid covering such that $m \in \cQ$ for all $\cQ \in \fH$.
    \begin{enumerate}
        \item Take $\cX$ and each $\cX_\cQ$ to be the set of fp-upset modules. Then $(\cX, \{\cX_\cQ\}_{\cQ \in \fH})$ is an $(\fH,\cP)$-extended projective class. For simplicity, restrict to the case where $n = 2$ and suppose every $\cQ = \cT_1' \times \cT_2' \in \fH$ satisfies $|\cT_1'| = |\cT_2'| \geq 3$. In \cite[Theorem~7.2]{BOOS}, it is shown that $\dim_{\cX_\cQ} = \sqrt{|\cQ|} - 2$ for all $\cQ \in \fH$. Thus if $\cP$ is infinite, Theorem~\ref{thm:extended_class} implies that the $\add(\cX)$-dimension of $\field_\cP$ is infinite, but not realisably infinite.
        \item Take $\cX$ and each $\cX_\cQ$ to be the set of fp-spread modules. Then $(\cX, \{\cX_\cQ\}_{\cQ \in \fH})$ is an $(\fH,\cP)$-extended projective class. Again restrict to the case where $n = 2$. 
        Then \cite[Proposition~4.5]{AENY} says that $\dim_{\cX_\cQ}(\field_\cQ)$ is finite for all $\cQ \in \fH$. For $\cP$ infinite, Example~\ref{ex:spreads_infinite} thus implies that the $\add(\cX)$-global dimension of $\field_\cP$ is infinite, but not realisably infinite. Based on discussion with Luis Scoccola \cite{luis}, it may also be possible to deduce this from \cite[Theorem~7.2]{BOOS} using a similar argument to Example~\ref{ex:spreads_infinite}.
    \end{enumerate}
\end{example}

Our final example highlights some subtlety regarding Definition~\ref{def:extended_class}(5) and its impact on Theorem~\ref{thm:extended_class}.

\begin{example}\label{ex:upsets}
    \begin{enumerate}
    \item Let $\cP = \mathbb{R}^2$ and let $\cX$ be the set of fp-upset modules. Then $\add(\cX)$-precovers do not always exist. Indeed, let $r,s,t \in \mathbb{R}$ with $r < s$ and let $X = \mathbb{I}_{\langle (r,t), (s,t)\langle}$. Let $\mathcal{A}$ be the set of fp-upset modules $\mathbb{I}_{\langle A,\infty\langle}$ which satisfy both (i) $(r,t) \in A$, and (ii) $\Pi_1(a) \geq s$ for all $a \in A \setminus (r,t)$. Proposition~\ref{prop:upset} readily generalizes to imply that $\mathcal{A}$ consists of precisely those fp-upset modules $\mathbb{I}_{\langle A,\infty\langle}$ which admit morphisms $h^A: \mathbb{I}_{\langle A,\infty\langle} \rightarrow X$ which are supported at $(r,t)$. Moreover, these morphisms factor through one another following the inclusion order of the corresponding upsets. Since the upsets appearing in $\mathcal{A}$ have no maximal element under inclusion, it follows that $X$ does not admit an $\add(\cX)$-precover.
    \item Let $\cP = \{1,2\}\times \{1,2\}$ with the usual product order, let $\cQ = \{(2,1),(2,2)\}$, and let $M = \mathbb{I}_{\langle (2,1),(2,1)\rangle} \in \rep(\cQ)$. Let $\cX$ and $\cX_\cQ$ be the sets of (fp-)upset modules. Then the (standard) projective cover $P = \langle (2,1),(2,1)\rangle \cdot \field_\cQ \rightarrow M$ is an $\add(\cX_{\cQ})$-cover of $M$ (in $\rep(\cQ)$). On the other hand, there is a morphism $\mathbb{I}_{\langle\{(1,2),(2,1)\},\infty\langle} \rightarrow \lceil M\rceil_\cQ$ in $\rep(\cP)$ which does not factor through $\langle (2,1),(2,1)\rangle \cdot \field_\cP = \lceil P\rceil_\cQ$. We conclude that $\lceil -\rceil_\cQ$ does not send $\cF_{\cX_\cQ}$-exact sequence to $\cF_\cX$-exact sequences. This is because Definition~\ref{def:extended_class}(5) does not hold for this choice of $\cQ$.
    \end{enumerate}
\end{example}

It remains an open problem to understand the meaning of Example~\ref{ex:upsets} in terms of the ``limit exact structure'' of \cite[Section~7]{BOOS}.


\section*{Acknowledgements}
The authors are thankful to Claire Amiot, Justin Desrochers, Steve Oudot, and Luis Scoccola for discussions on portions of this manuscript. This work was partially supported by the Canada Research Chairs program (CRC-2021-00120) and NSERC Discovery Grants (RGPIN-2022-03960 and RGPIN/04465-2019).


\end{document}